\documentclass[ejs]{imsart}

\usepackage{amsthm,amsmath}
\RequirePackage[numbers]{natbib}
\RequirePackage[colorlinks,citecolor=blue,urlcolor=blue]{hyperref}

\startlocaldefs
\usepackage{amsfonts,amssymb,url,enumitem,mathtools}

\DeclarePairedDelimiterX{\infdivx}[2]{(}{)}{%
	#1\;\delimsize\|\;#2%
}
\newcommand{\infdiv}{\mathrm{KL}\infdivx}
\newcommand{\sinfdiv}{\mathrm{J}\infdivx}

\DeclarePairedDelimiter{\norm}{\lVert}{\rVert}

\newcommand{\bS}{\mathbb{S}}
\newcommand{\sP}{\mathcal{P}}
\newcommand{\sW}{\mathcal{W}}
\newcommand{\sG}{\mathcal{G}}
\newcommand{\sS}{\mathcal{S}}
\newcommand{\sI}{\mathcal{I}}
\newcommand{\sX}{\mathcal{X}}
\newcommand{\sF}{\mathcal{F}}
\newcommand{\sR}{\mathcal{R}}
\newcommand{\sA}{\mathcal{A}}
\newcommand{\sM}{\mathcal{M}}
\newcommand{\sN}{\mathcal{N}}
\newcommand{\sB}{\mathcal{B}}

\newcommand{\R}{\mathbb{R}}

\DeclareMathOperator{\E}{\mathbf{E}}

\newcommand{\eps}{\varepsilon}
\newcommand{\tSigma}{\tilde{\Sigma}}

\let\Pr\relax
\DeclareMathOperator{\Pr}{\mathbf{P}}

\DeclareMathOperator{\TV}{TV}
\DeclareMathOperator{\tr}{tr}
\DeclareMathOperator{\VC}{VC}

\newcommand{\widesim}[2][1.5]{
	\mathrel{\overset{#2}{\scalebox{#1}[1]{$\sim$}}}
}
\newcommand{\iid}{\mathbin{\widesim[2]{i.i.d.}}}

\let\originalleft\left
\let\originalright\right
\def\left#1{\mathopen{}\originalleft#1}
\def\right#1{\originalright#1\mathclose{}}

\newcommand{\dif}{\mathop{}\!\mathrm{d}}

\newcommand{\seclabel}[1]{\label{sec:#1}}

\newcommand{\secref}[1]{\mbox{Section~\ref{sec:#1}}}

\newcommand{\eqlabel}[1]{\label{eq:#1}}
\renewcommand{\eqref}[1]{(\ref{eq:#1})}

\newtheorem{thm}{Theorem}{\bfseries}{\itshape}
\newcommand{\thmlabel}[1]{\label{thm:#1}}
\newcommand{\thmref}[1]{Theorem~\ref{thm:#1}}
\numberwithin{thm}{section}

\newtheorem{lem}[thm]{Lemma}{\bfseries}{\itshape}
\newcommand{\lemlabel}[1]{\label{lem:#1}}
\newcommand{\lemref}[1]{Lemma~\ref{lem:#1}}

\newtheorem{cor}[thm]{Corollary}{\bfseries}{\itshape}

\newtheorem{prop}[thm]{Proposition}{\bfseries}{\itshape}
\newcommand{\proplabel}[1]{\label{prop:#1}}
\newcommand{\propref}[1]{Proposition~\ref{prop:#1}}

\theoremstyle{definition}

\theoremstyle{plain}

\renewcommand{\tilde}[1]{\widetilde{#1}}
\newcommand{\transpose}{^{\mathsf{T}}}

\endlocaldefs

\begin{document}

\begin{frontmatter}
\title{The minimax learning rates of normal and Ising undirected graphical models}
\runtitle{Minimax Learning Rates of Normal \& Ising MRF}


 \author{\fnms{Luc} \snm{Devroye}\thanksref{t2}\ead[label=e1]{lucdevroye@gmail.com}}
 \thankstext{t2}{Supported by NSERC Grant A3456.} 
\and
 \author{\fnms{Abbas} \snm{Mehrabian}\thanksref{t3}\corref{}\ead[label=e2]{abbas.mehrabian@gmail.com}}
\thankstext{t3}{Supported by a CRM-ISM postdoctoral fellowship and an IVADO-Apog\'ee-CFREF postdoctoral fellowship.} 
\and
 \author{\fnms{Tommy} \snm{Reddad}\thanksref{t4}\ead[label=e3]{tommy.reddad@gmail.com}}
\thankstext{t4}{Supported by an NSERC PGS D scholarship 396164433.}

\address{School of Computer Science \\ McGill University\\ Montr\'{e}al, Qu\'{e}bec, Canada \\ \printead{e1,e2,e3}}

\runauthor{Devroye, Mehrabian, and Reddad}

\begin{abstract}
	{Let $G$ be an undirected graph with $m$ edges and $d$ vertices. We
	show that $d$-dimensional Ising models on $G$ can be learned from
	$n$ i.i.d.\ samples within expected total variation distance some
	constant factor of $\min\{1, \sqrt{(m + d)/n}\}$, and that this rate
	is optimal. We show that the same rate holds for the class of
	$d$-dimensional multivariate normal undirected graphical models with
	respect to $G$. We also identify the optimal rate of
	$\min\{1, \sqrt{m/n}\}$ for Ising models with no external magnetic
	field.}
\end{abstract}




\tableofcontents

\end{frontmatter}

\section{Introduction}
The Ising model is a popular mathematical model inspired by
ferromagnetism in statistical mechanics.  The model consists of
discrete $\{-1,1\}$ random variables representing magnetic dipole
moments of atomic spins.  The spins are arranged in a
graph---originally a lattice, but other graphs have also been
considered---allowing each spin to interact with its graph
neighbors. Sometimes, the spins are also subject to an external
magnetic field.

The Ising model is one of many possible mean field models for spin
glasses. Its probabilistic properties have caught the attention of
many researchers---see, e.g., the monographs of
Talagrand~\cite{talagrand-2003,talagrand-2010,talagrand-2011}. The
analysis of social networks has brought computer scientists into the
fray, as precisely the same model appears there in the context of
community detection~\cite{berthet}.

In this work we view an Ising model as a probability distribution on
$\{-1,1\}^d$, and consider the following statistical inference and
learning problem, known as {\em density estimation} or {\em
	distribution learning}: given i.i.d.\ samples from an unknown Ising
model $I$ on a known graph $G$, can we create a probability
distribution on $\{-1,1\}^d$ that is close to $I$ in total variation
distance?  If we have $n$ samples, then how small can we make the
expected value of this distance? We prove that if $G$ has $m$ edges,
the answer to this question is bounded from above and below by
constant factors of $\sqrt{(m + d)/n}$. In the case when there is no
external magnetic field, the answer is instead $\sqrt{m/n}$.

Our techniques carry over to the continuous case and allow us to prove a
similar minimax risk for learning the class of $d$-dimensional normal
undirected graphical models on $G$. It is surprising that the minimax
rate for this class was not known, even when $G$ is the complete
graph, corresponding to the class of all $d$-dimensional normal
distributions.

\subsection{Main results}

We start by stating our results for normal distributions.
For formal
definitions of all terms mentioned below, see \secref{prelim}.
We express minimax risks as functions of the number of samples $n$.

\begin{thm}[Main result for learning normals]\thmlabel{normal-main-bound}
	Let $G$ be a given undirected graph with vertex set
	$\{1, \dots, d\}$ and $m$ edges. Let $\sF_G$ be the class of
	$d$-dimensional multivariate normal undirected graphical models with
	respect to $G$. Then, the minimax rate for learning $\sF_{G}$ in
	total variation distance is bounded from above and below by constant
	factors of $\min\{1, \sqrt{(m + d)/n}\}$.
\end{thm}

The upper bound follows from standard techniques (see
\secref{normal-upper}) and a lower bound of $\min\{1, \sqrt{d/n}\}$ is
known (see Section~\ref{sec:related}); our main technical contribution
is to show a lower bound of $\min\{1, \sqrt{m/n}\}$, from which
\thmref{normal-main-bound} follows.  This theorem immediately implies
a tight result on the minimax rate for learning the class of all
$d$-dimensional normals, if we take the graph $G$ to be complete. In
this specific case, the upper bound is already known, so our
contribution is the matching lower bound.
\begin{cor}
	The minimax rate for learning the class of all $d$-dimensional
	multivariate normal distributions in total variation distance is
	bounded from above and below by constant factors of
	$\min\{1, d/\sqrt{n}\}$.
\end{cor}
In fact, this result can be extended using the techniques of
\cite{2017-abbas} to yield the lower bound of 
$\Omega(\min\{1, d\sqrt{k/n}\})$ for the minimax learning rate of \emph{mixtures} of $k$ many
$d$-dimensional multivariate normals, which was previously known only
up to logarithmic factors.

We remark that for the class of \emph{zero-mean} normal undirected graphical
models, we prove a lower bound of $\min\{1, \sqrt{m/n}\}$ while the
best known upper bound is $\min\{1, \sqrt{(m + d)/n}\}$. In practice,
the underlying graph is typically connected, which means that
$m \ge d - 1$, so these bounds match.

We prove similar rates as in \thmref{normal-main-bound} for the class
of Ising models, which resemble discrete versions of multivariate
normal distributions. An Ising model in dimension $d$ is supported on
$\{-1, 1\}^d$ and comes with an undirected graph $G = (V, E)$ with
vertex set $V = \{1, \dots, d\}$, edge set
$E \subseteq \{\{i, j\} \colon i \neq j \in V\}$, interactions
$w_{ij} \in \R$ for each $\{i, j\} \in E$, and external magnetic field
$h_i \in \R$ for $1 \le i \le d$ such that $x \in \{-1, 1\}^d$ appears
with probability proportional to
\[
\exp\left\{\sum_{\{i, j\} \in E} w_{ij} x_i x_j + \sum_{i = 1}^d h_i x_i\right\} .
\]
Note that our definition has no temperature parameter; we have
absorbed it into the weights.

\begin{thm}[Main result for learning Ising models]\thmlabel{ising-main-bound}
	Let $G$ be a given undirected graph with vertex set
	$\{1, \dots, d\}$ and $m$ edges. Let $\sI_G$ be the class of
	$d$-dimensional Ising models with underlying graph $G$.
	\begin{enumerate}[label=(\roman*)]
		\item \label{ising-1} The minimax rate for learning $\sI_{G}$ in
		total variation distance is bounded from above and below by
		constant factors of $\min\{1, \sqrt{(m + d)/n}\}$.
		\item \label{ising-2} Let $\sI'_{G}$ be the subclass $\sI_G$ of
		Ising models with no external magnetic field. The minimax rate for
		learning $\sI'_{G}$ in total variation distance is bounded from
		above and below by constant factors of $\min\{1, \sqrt{m/n}\}$.
	\end{enumerate}
\end{thm}

In all of the above cases, it is assumed that the full structure and labeling of the
underlying graph $G$ is known in advance. We next consider the case in
which it is only known that the underlying graph has $d$ vertices and
$m$ edges.
\begin{thm}\thmlabel{unknowngraph}
	Let $\sF_{d,m}$ and $\sI_{d,m}$ be the class of all normal and Ising
	undirected graphical models with respect to some unknown graph with
	$d$ vertices and $m$ edges. The minimax learning rates for
	$\sF_{d, m}$ and $\sI_{d, m}$ are both bounded from above by a
	constant factor of $\min\{1, \sqrt{(m + d) \log (d)/n}\}$, and bounded
	from below by a constant factor of $\min\{1, \sqrt{(m + d)/n}\}$.
\end{thm}
The lower bound in this theorem follows immediately from our lower
bounds for the case in which the graph is known.

In the next section we review related work.  In \secref{prelim} we
discuss preliminaries.  \thmref{normal-main-bound},
\thmref{ising-main-bound}, and \thmref{unknowngraph} are proved in
\secref{normal}, \secref{ising}, and Section~\ref{sec:unknowngraph},
respectively.  We conclude with some open problems in
\secref{conclusion}.

\subsection{Related work}\label{sec:related}

Density estimation is a central problem in statistics and has a long
history~\cite{devroye-course,devroye-gyorfi,ibragimov,tsybakov}.  It
has also been studied in the learning theory community under the name
\emph{distribution learning}, starting from \cite{Kearns}, whose focus
is on the computational complexity of the learning problem. Recently,
it has gained a lot of attention in the machine learning community, as
one of the important tasks in unsupervised learning is to understand
the distribution of the data, which is known to significantly improve
the efficiency of learning algorithms (e.g.,
\cite[page~100]{deeplearningbook}).  See~\cite{Diakonikolas2016} for a
recent survey from this perspective.

An upper bound on the order of $d/\sqrt{n}$ for estimating
$d$-dimensional normals can be obtained via empirical mean and
covariance estimation (e.g., \cite[Theorem~B.1]{2017-abbas}) or via
Yatracos' techniques based on VC-dimension (e.g.,
\cite[Theorem~13]{abbas-mixtures}).  Regarding lower bounds, Acharya,
Jafarpour, Orlitsky, and Suresh~\cite[Theorem~2]{acharya-lower-bound}
proved a lower bound on the order of $\sqrt{d/n}$ for spherical
normals (i.e., normals with identity covariance matrix), which implies
the same lower bound for general normals. The lower bound for general
normals was recently improved to a constant factor of
$\frac{d}{\sqrt{n} \log n}$ by Asthiani, Ben-David, Harvey, Liaw,
Mehrabian, and Plan~\cite{2017-abbas}. In comparison, our result
shaves off the logarithmic factor. Moreover, their result is
nonconstructive and relies on the probabilistic method, while our
argument is fully deterministic.
(Recently, Ashtiani et al.~\cite{2020-abbas} have also shaved off the logarithmic factor and have shown that their probabilistic construction achieves the optimal lower bound of $\frac{d}{\sqrt{n}}$ as well.)

For the Ising model, the main focus in the literature has been on
learning the structure of the underlying graph rather than learning
the distribution itself, i.e., how many samples are needed to
reconstruct the underlying graph with high probability?
Bresler~\cite{Bresler} and Vuffray, Misra, Lokhov, and Chertkov~\cite{VMLC} provided efficient algorithms for this task on bounded-degree graphs.
For general graphs, see~\cite{Santhanam-lower,Shanmugam-lower} for some lower bounds
and~\cite{Hamilton-upper,Klivans-upper} for some upper bounds. Klivans
and Meka~\cite{Klivans-upper} also give an efficient algorithm for
learning all of the parameters of an Ising model up to an additive error given a natural
parametric constraint. Otherwise, the Ising model itself has been
studied by physicists in other settings for nearly a century. See the
books of Talagrand for a comprehensive look at the mathematics of the
Ising model~\cite{talagrand-2003, talagrand-2010, talagrand-2011}.

Daskalakis, Dikkala, and Kamath~\cite{costis-2018} were the first to
study Ising models from a statistical point of view. However, their
goal is to test whether an Ising model has certain properties, rather
than estimating the model, which is our goal. Moreover, their focus is
on designing efficient testing algorithms. They prove polynomial
sample complexities and running times for testing various properties
of the model.

An alternative goal would be to estimate the parameters of the
underlying model (see, e.g., \cite{KMV}) rather than coming up with a model
that is statistically close, which is our focus.  We remark that these
two goals are quantitatively different, although similar techniques
may be used for both. In general, estimating the parameters of a model
to within some accuracy does not necessarily result in a distribution
that is close to the original distribution in a statistical sense. For
instance, define
\[
\Sigma = \begin{pmatrix} 1 & -0.99 \\ -0.99 & 1 \end{pmatrix}
\qquad\text{and}\qquad
\tilde{\Sigma} = \begin{pmatrix} 1 & -1 \\ -1 & 1 \end{pmatrix},
\]
and observe that $\Sigma$ and $\tilde{\Sigma}$ are entrywise very
close. However, $\Sigma$ is non-singular and $\tilde{\Sigma}$ is
singular, and thus two zero-mean normal distributions with covariance
matrices $\Sigma$ and $\tilde{\Sigma}$ are at total variation distance
$1$ from one another. Conversely, if two distributions are close in
total variation distance, their parameters are not necessarily close
to within the same accuracy (see, e.g., \cite[Section~1.3.1]{2017-abbas}).

\section{Preliminaries}\seclabel{prelim}
The goal of density estimation is to design an estimator $\hat{f}$ for
an unknown function $f$ taken from a known class of functions
$\sF$. In the continuous case, $\sF$ is a class of probability density
functions with sample space $\sX = \R^d$ for some $d \ge 1$; in the
discrete case, $\sF$ is a class of probability mass functions with a
countable sample space $\sX$. In either case, in order to create the
estimator $\hat{f}$, we have access to samples
$X_1, \dots, X_n \iid f$. Our measure of closeness is the \emph{total
	variation (TV) distance}: For functions $f, g : \sX \to \R$, their
TV-distance is defined as
\begin{align*}
	\TV(f, g) =  \norm{f - g}_1 /2,
\end{align*}
where for any function $f$, the $L^1$-norm of $f$  is defined as
\begin{align*}
	\norm{f}_1 &= \int_{\sX} |f(x)| \dif x &&\textrm{in the continuous case, and} \\
	\norm{f}_1 &= \sum_{x \in \sX} |f(x)| &&\textrm{in the discrete case. }
\end{align*}
Further along, we will also need the \emph{Kullback-Leibler (KL)
	divergence} or \emph{relative entropy}~\cite{kullback-book}, which
is another measure of closeness of distributions defined by
\begin{align*}
	\infdiv{f}{g} &= \int_{\sX} f(x) \log\left( \frac{f(x)}{g(x)} \right) \dif x &&\textrm{in the continuous case, and} \\
	\infdiv{f}{g} &= \sum_{x \in \sX} f(x) \log \left( \frac{f(x)}{g(x)} \right) &&\textrm{in the discrete case.}
\end{align*}
Formally, in the continuous case, we can write
$f = \frac{\dif F}{\dif \mu}$ for a probability measure $F$ and $\mu$
the Lebesgue measure on $\R^d$, and in the discrete case
$f = \frac{\dif F}{\dif \mu}$ for a probability measure $F$ and $\mu$
the counting measure on countable $\sX$. In view of this unified
framework, we say that $\sF$ is a \emph{class of densities} and that
$\hat{f}$ is a \emph{density estimate}, in both the continuous and the
discrete settings.  

The total variation distance has a natural
probabilistic interpretation as
$\TV(f, g) = \sup_{A\subseteq \sX} |F(A)-G(A)|$, where $F$ and $G$ are
probability measures corresponding to $f$ and $g$, respectively. So,
the TV-distance lies in $[0,1]$.  Also, it is well known that the
KL-divergence is nonnegative, and is zero if and only if the two
densities are equal almost everywhere.  However, it is not symmetric
in general, and can become $+\infty$.

For density estimation there are various possible measures of distance
between distributions.  Here we focus on the TV-distance since it has
several appealing properties, such as being a metric and having a
natural probabilistic interpretation.  For a detailed discussion on
why TV is a natural choice, see \cite[Chapter~5]{devroye-lugosi}.  If
$\hat{f}$ is a density estimate, we define the \emph{risk} of the
estimator $\hat{f}$ with respect to the class $\sF$ as
\[
\sR_n(\hat{f}, \sF) = \sup_{f \in \sF} \E\{ \TV(\hat{f}, f) \} ,
\]
where the expectation is over the $n$ i.i.d.\ samples from $f$, and
possible randomization of the estimator.  The \emph{minimax risk} or
\emph{minimax rate} for learning $\sF$ is the smallest risk over all possible
estimators,
\[
\sR_n(\sF) = \inf_{\hat{f} \colon \sX^n \to \R^{\sX}} \sR_n(\hat{f}, \sF) .
\]

For a class of functions $\sF$ defined on the same domain $\sX$, its
\emph{Yatracos class} $\sA$,
first defined in~\cite{yatracos},
is the class of sets defined by
\[
\sA = \Big\{\{x \in \sX \colon f(x) > g(x)\} \colon f \neq g \in \sF\Big\}.
\]
The following powerful result relates the minimax risk of a class of
densities to an old well-studied combinatorial quantity called the
Vapnik-Chervonenkis (VC) dimension~\cite{vapnik-cherv}, defined next. Let
$\sA \subseteq 2^{\sX}$ be a family of subsets of $\sX$.  The
\emph{VC-dimension} of $\sA$, denoted by $\VC(\sA)$, is the size of
the largest set $X \subseteq \sX$ such that for each $Y\subseteq X$
there exists $B \in \sA$ such that $X \cap B = Y$.  See, e.g.,
\cite[Chapter~4]{devroye-lugosi} for examples and applications.
\begin{thm}[\protect{\cite[Section~8.1]{devroye-lugosi}}]\thmlabel{risk-vc}
	There is a universal constant $c > 0$ such that for any class of
	densities $\sF$ with Yatracos class $\sA$,
	\[
	\sR_n(\sF) \le c \sqrt{\VC(\sA)/n} .
	\]
\end{thm}

On the other hand, there are several methods for obtaining lower
bounds on minimax risk; we emphasize, in particular, the methods of
Assouad~\cite{assouad}, Le Cam~\cite{lecam-1, lecam-2}, and
Fano~\cite{has-fano}. Each of these involve picking a finite subclass
$\sG \subseteq \sF$, and developing a lower bound on the minimax
risk of $\sG$
using the fact that
$\sR_n(\sG) \le \sR_n(\sF)$---see \cite{devroye-course, devroye-lugosi, yu-survey}
for examples. We will use the following result, known as
(generalized) Fano's lemma, originally due to
Has'minski\u\i~\cite{has-fano}, where $\log(\cdot)$ denotes the natural logarithm function.

\begin{lem}[Fano's Lemma \protect{\cite[Lemma
		3]{yu-survey}}]\lemlabel{fano-lower}
	Let $\sF$ be a finite  class of densities such that
	\begin{align*}
		\inf_{f \neq g \in \sF} \norm{f - g}_1 \ge \alpha , \qquad \sup_{f \neq g \in \sF} \infdiv{f}{g} \le \beta .
	\end{align*}
	Then,
	\[
	\sR_n(\sF) \ge \frac{\alpha}{4} \left(1 - \frac{n \beta + \log 2}{\log |\sF|} \right) .
	\]
\end{lem}
In light of this lemma, to prove a minimax risk lower bound on a class
of densities $\sF$, we shall carefully pick a finite subclass of
densities in $\sF$, such that any two densities in this subclass are
far apart in $L^1$-distance but close in KL-divergence.

Throughout this paper, we will be estimating densities from classes
with a given graphical dependence structure, known as undirected
graphical models~\cite{graphical-models}. The underlying graph will
always be undirected and without parallel edges or self-loops, so we
will omit these qualifiers henceforth. Let $G = (V, E)$ be a
given graph with vertex set $V = \{1, \dots, d\}$ and edge set $E$. A
set of random variables $\{X_1, \dots, X_d\}$ with everywhere strictly
positive densities forms a \emph{graphical model} or \emph{Markov
	random field (MRF)} with respect to $G$ if, for every
$\{i, j\} \not\in E$, the variables $X_i$ and $X_j$ are conditionally
independent given $\{X_k \colon k \neq i, j\}$.

Often, the problem of density estimation is framed slightly
differently than we have presented it: given $\eps \in (0, 1)$, one might
be interested in finding the smallest number of i.i.d.\ samples
$m_\sF(\eps)$ for which there exists a density estimate $\hat{f}$
based on these samples satisfying
$\sup_{f \in \sF} \E\{\TV(\hat{f}, f)\} \le \eps$. Or, given
$\delta \in (0, 1)$, one might want to find the minimum number of
samples $m_\sF(\eps, \delta)$ for which there is a density estimate
$\hat{f}$ satisfying $\sup_{f \in \sF} \TV(\hat{f}, f) \le \eps$ with
probability at least $1 - \delta$. The quantities $m_\sF(\eps)$ and
$m_\sF(\eps, \delta)$ are known as \emph{sample complexities} of the
class $\sF$. Note that $m_\sF(\eps)$ and $\sR_n(\sF)$ are related
through the equation
\[
m_\sF(\eps)
= \min \{n\ge0 : \sR_n(\sF) \leq \eps \},
\]
so determining $\sR_n(\sF)$ also determines $m_\sF(\cdot)$. Moreover, $\delta$
is often fixed to be some small constant, say $1/3$, when studying
$m_\sF(\eps, \delta)$. Then, there are versions of \thmref{risk-vc} and
\lemref{fano-lower} for $m_\sF(\eps, 1/3)$, which introduce some
extraneous $\log(1/\eps)$ factors. In order to avoid such extraneous
logarithmic factors, we focus on $\sR_n(\sF)$---equivalently,
$m_\sF(\eps)$---rather than $m_\sF(\eps, 1/3)$ or
$m_\sF(\eps, \delta)$.

We now recall some basic matrix analysis formulae that will be used
throughout (see Horn and Johnson~\cite{matrix_analysis} for the
proofs). 
We denote the identity matrix by $I$.
For a matrix $A = (A_{ij}) \in \R^{d \times d}$, the
\emph{spectral norm} of $A$ is defined as
$\norm{A} = \sup_{x \in \bS^{d - 1}} \norm{A x}_2$, where
$\bS^{d - 1} = \{ x \in \R^d \colon \|x\|_2 = 1\}$ is the unit
$(d - 1)$-sphere. 
The trace and determinant of $A$ are denoted by $\tr(A)$ and $\det(A)$, respectively.
Recall also the \emph{Frobenius norm} of $A$,
also called the \emph{Hilbert-Schmidt norm} or the \emph{Schur norm},
$\norm{A}_F = \sqrt{ \sum_{i,j = 1}^d A_{i j}^2} =
\sqrt{\tr(A\transpose A)}$. When $A$ has only real eigenvalues, we
write $\lambda_i(A)$ for the $i$th largest eigenvalue of $A$. In
general, we write $\sigma_i(A) = \sqrt{\lambda_i(A\transpose A)}$ for
the $i$th largest singular value of $A$. Then,
$\det(A) = \prod_{i = 1}^d \lambda_i(A)$, and for any $k \ge 1$,
$\tr(A^k) = \sum_{i = 1}^d \lambda_i^k(A)$. Furthermore,
$\|A\| = \sigma_1(A)$, and
$\|A\|_F = \sqrt{\sum_{i = 1}^d \sigma_i(A)^2}$, so
$\|A\| \le \|A\|_F$. 
When
$A$ is invertible, $\sigma_i(A^{-1}) = \sigma_{d - i}(A)^{-1}$ for
every $1 \le i \le d$.
Finally, for any matrix $B \in \R^{d \times d}$, we have (see, e.g., \cite[Fact~7(c) in Section~24.4]{linearalgebrahandbook})
\begin{equation}
	\max\{ 
	\sigma_d(A) \|B\|_F,
	\sigma_d(B) \|A\|_F
	\} \le
	\|A B\|_F \le 
	\min\{\sigma_1(A) \|B\|_F, \sigma_1(B) \|A\|_F\}.\label{frobeqs}
\end{equation}

Throughout this paper, we let $c_0,c_1, c_2, \ldots \in \R$ denote
positive absolute constants. We liberally reuse these symbols, i.e.,
every $c_i$ may differ between proofs and statements of different
results.  We denote the set $\{1,\dots,d\}$ by $[d]$.

\section{Learning normal graphical models}\seclabel{normal}
Let $d$ be a positive integer, $\sP_d \subseteq \R^{d \times d}$ be
the set of positive definite $d \times d$ matrices over $\R$, and
$\sN(\mu, \Sigma)$ denote the multivariate normal distribution with
mean $\mu \in \R^d$, covariance matrix $\Sigma \in \sP_d$, and
corresponding probability density function $f_{\mu, \Sigma}$, where
for $x \in \R^d$,
\[
f_{\mu, \Sigma}(x) = \frac{\exp\left\{ - \frac{1}{2} (x - \mu)\transpose \Sigma^{-1} (x - \mu) \right\} }{(2 \pi)^{d/2} \sqrt{\det(\Sigma)}} .
\]
Let $G = ([d], E)$ be a given graph with $m$ edges. Let
$\sP_G \subseteq \sP_d$ be the following subset of all positive
definite matrices,
\[
\sP_G = \Big\{ \Sigma \in \sP_d \colon \text{ if } \{i, j\} \not\in E \text{ with } i \neq j \in [d], \text{ then } \Sigma^{-1}_{ij} = 0 \Big\} .
\]
The main result of this section is a characterization of the minimax
risk of
\[
\sF_G = \left\{ f_{\mu, \Sigma} \colon \mu \in \R^d , \, \Sigma \in \sP_G \right\} .
\]
It is known that $\sF_G$ is precisely the class of $d$-dimensional
multivariate normal graphical models with respect to
$G$~\cite[Proposition~5.2]{graphical-models}.

\subsection{Proof of the upper bound in \thmref{normal-main-bound}}\seclabel{normal-upper}
We can already prove the upper bound in \thmref{normal-main-bound}
without lifting a finger.  The proof is similar to that of
\cite[Theorem 13]{abbas-mixtures}, which gives the optimal upper bound on the
minimax risk of all multivariate normals, corresponding to the special case in
which $G$ is complete.  Let $\sA$ be the Yatracos class of $\sF_G$,
\begin{align*}
	\sA = \Big\{\{x \in \R^d \colon f_{\mu,\Sigma}(x) > f_{\tilde{\mu},\tilde{\Sigma}}(x)\} \colon (\mu,\Sigma) \neq (\tilde{\mu},\tilde{\Sigma}) \in \R^d \times \sP_d \Big\} ,
\end{align*}
which, after taking logarithms and simplification, is easily seen to be contained in the
larger class
\[
\sA' = \Big\{\{x \in \R^d \colon x\transpose A x + b\transpose x +c > 0\} \colon A \in \R^{d \times d}, \, b \in \R^d, \, c \in \R \Big\},
\]
and thus $\VC(\sA) \le \VC(\sA')$. It remains to upper-bound
$\VC(\sA')$.

In general, let $\sG$ be a vector space of real-valued functions, and define
$\sB \coloneqq \{\{x \colon f(x) > 0\} \colon f \in
\sG\}$. Dudley~\cite[Theorem~7.2]{dudley_vectorspace} proved that
$\VC(\sB) \le \dim(\sG)$. (See \cite[Lemma~4.2]{devroye-lugosi} for a
historical discussion.) In our case, the vector space $\sG$ has a
basis of monomials
\[
\{1\} \cup \{x_i x_j \colon \{i, j\} \in E\} \cup \{x_i,x_i^2 \colon i\in[d]\} ,
\]
so $\VC(\sA') \le m + 2d+1$. By \thmref{risk-vc}, there is a universal
constant $c > 0$ such that
\[
\sR_n(\sF_G) \le c \sqrt{\frac{\VC(\sA')}{n}} \le c \sqrt{\frac{m + 2d+1}{n}} ,
\]
while the upper bound $\sR_n(\sF_G) \leq 1$ follows simply because the
TV-distance is bounded by $1$. \qed

\subsection{Proof of the lower bound in \thmref{normal-main-bound}}\seclabel{normal-lower}
Since a lower bound on the order of $\min\{1,\sqrt{d/n}\}$ for
spherical normals was proved in~\cite[Theorem~2]{acharya-lower-bound},
the lower bound in \thmref{normal-main-bound} follows from
subadditivity of the square root after the following proposition.
\begin{prop}\proplabel{normal-lower-bound}
	There exist $c_0,c_1> 0$ such that for any graph $G = ([d], E)$
	with $m$ edges, where $n \ge c_1 m$,
	\[
	\sR_n(\sF_G) \ge c_0 \sqrt{m/n} .
	\]
\end{prop}
Note that if $n < c_1 m$, then
$\sR_n(\sF_G) \geq \sR_{c_1 m}(\sF_G) \ge c_0 \sqrt{1/c_1}$, which
implies the lower bound in \thmref{normal-main-bound} 
for such $n$. We prove \propref{normal-lower-bound} via
\lemref{fano-lower}.  This involves choosing a finite subset of
$\sF_G$.  Our normal densities will be zero-mean, but the covariance
matrices will be chosen carefully.  To make this choice, we use the
next result which follows from an old theorem of
Gilbert~\cite{gilbert} and independently Varshamov~\cite{varshamov}
from coding theory.
\begin{thm}\thmlabel{gilbert-varshamov}
	There is a subset $Q \subseteq \{-1, 1\}^m$ of size at least
	$2^{m/5}$ such that for any distinct $s, \tilde{s} \in Q$ we have
	$\norm{s - \tilde{s}}_1 \ge m/3$.
\end{thm}
\begin{proof}
	We give an iterative algorithm to build $Q$: choose a vertex from
	the hypercube, put it in $Q$, remove the hypercube points in the
	corresponding $L^1$-ball of radius $m/3$, and repeat.  Since the
	intersection of this ball and the hypercube has size at most
	\[
	\sum_{i=0}^{m/6} \binom{m}{i} \leq \left(\frac {em}{m/6}\right)^{m/6} = (6e)^{m/6} < 2^{4m/5},
	\]
	the size of the final set $Q$ will be at least $2^{m/5}$. Note that the sum goes up to $m/6$ and not $m/3$ here, because we are working in the hypercube $\{-1,1\}^m$, hence two vertices with $m/6$ different coordinates have $L^1$ distance $m/3$.
\end{proof}

Let $\sS \subseteq \{-1, 1\}^{m}$ be as in \thmref{gilbert-varshamov},
so that $|\sS| \ge 2^{m/5}$ and for any distinct
$s, \tilde{s} \in \sS$, $\norm{s - \tilde{s}}_1 \ge m/3$.  Let
$\delta > 0$ be a real number to be specified later. Enumerate the
edges of $G$ from $1$ to $m$, and for $s \in \sS$, set
$\Sigma(s)^{-1}$ to be the $d \times d$ matrix with entries
\[
\left(\Sigma(s)^{-1}\right)_{i j} = \left\{\begin{array}{ll}
1 & \mbox{if $i = j$,} \\
0 & \mbox{if $i \neq j$ and $\{i, j\} \not\in E$,} \\
\delta s_{\{i, j\}} & \mbox{if $i \neq j$ and $\{i, j\} \in E$.}
\end{array} \right.
\]
In other words, $\Sigma(s)^{-1}$ is symmetric with all ones on its
diagonal, $\pm \delta$ everywhere along the nonzero entries of the
adjacency matrix of $G$ according to the signs in $s$, and $0$
elsewhere.
\begin{lem}\lemlabel{psd}
	Suppose that $\delta^2 m \le 1/8$. Then, for any $s \in \sS$, the
	matrix $\Sigma(s)^{-1}$ is positive definite and its eigenvalues lie in 
	$[1/2,3/2]$.
\end{lem}
\begin{proof}
	Since $\Sigma(s)^{-1}$ is symmetric and real, all its eigenvalues
	are real. Write $\Sigma(s)^{-1} = I + \Delta$, so that
	$\lambda_i(\Sigma(s)^{-1}) = 1 + \lambda_i(\Delta)$. Observe that
	\begin{equation}
		|\lambda_i(\Delta)| \le \|\Delta\| \le \|\Delta\|_F \le \sqrt{2 \delta^2 m} \le 1/2 .
		\eqlabel{delta12}
	\end{equation}
	Then,
	$1/2 \leq \lambda_i(\Sigma(s)^{-1}) \leq 3/2$ for every $1 \le i \le d$, and so
	$\Sigma(s)^{-1}$ is positive definite.
\end{proof}
We will assume from now on that $\delta^2 m \le 1/8$.  In light of
\lemref{psd}, $\Sigma(s)^{-1}$ is positive definite, so it is
invertible, and we let $\Sigma(s)$ denote its inverse.  Since we will
always take the mean to be $0$, we will write $f_{\Sigma}$ for
$f_{0, \Sigma}$ from now on.  We define the set
$\sW = \{\Sigma(s) \colon s \in \sS\}$ of covariance matrices, and let
\[
\sF = \{f_{\Sigma} \colon \Sigma \in \sW\} .
\]
In order to prove \propref{normal-lower-bound} via
\lemref{fano-lower}, it suffices to exhibit upper bounds on the
KL-divergence between any two densities in $\sF$, and lower bounds on
their $L^1$-distances.
\begin{lem}\lemlabel{normal-kl}
	For any
	$\Sigma, \tilde{\Sigma} \in \sP_d$ satisfying
	$\max\{\|\Sigma^{-1} - I\|_F, \|\tilde{\Sigma}^{-1} - I\|_F\} \le 1/2$,
	\[
	\infdiv{f_{\Sigma}}{f_{\tilde{\Sigma}}} \le 2
	\|\tilde{\Sigma}^{-1} - \Sigma^{-1}\|_F^2.
	\]
\end{lem}
\begin{proof}
	We consider a symmetrized KL-divergence, often
	called the \emph{Jeffreys divergence}~\cite{kullback-book},
	\[
	\sinfdiv{f_{\Sigma}}{f_{\tilde{\Sigma}}} = \infdiv{f_{\Sigma}}{f_{\tilde{\Sigma}}} + \infdiv{f_{\tilde{\Sigma}}}{f_\Sigma} ,
	\]
	which clearly serves as an upper bound on the quantity of
	interest. It is well known that
	\[
	\sinfdiv{f_\Sigma}{f_{\tilde{\Sigma}}} =  \tr((\Sigma - \tilde{\Sigma}) (\tilde\Sigma^{-1} - {\Sigma}^{-1}))/2 ,
	\]
	e.g.,\ by \cite[Section 9.1]{kullback-book}. Since $\Sigma - \tilde{\Sigma}$ is symmetric, the inequality $\tr(A\transpose B) \leq \|A\|_F \cdot \|B\|_F$, which is the 
	Cauchy-Schwarz inequality for the inner product 
	$\langle A, B \rangle = \tr(A\transpose B)$, gives
	\[
	\sinfdiv{f_{\Sigma}}{f_{\tilde{\Sigma}}} \le \norm{\tilde{\Sigma}^{-1} - \Sigma^{-1}}_F \norm{\Sigma - \tilde{\Sigma}}_F/2 .
	\]
	Notice now that
	$\Sigma - \tilde{\Sigma} = \Sigma (\tilde{\Sigma}^{-1} -
	\Sigma^{-1}) \tilde{\Sigma}$, so by~(\ref{frobeqs}),
	\begin{align*}
		\norm{\Sigma - \tilde{\Sigma}}_F &= \norm{\Sigma (\tilde{\Sigma}^{-1} - \Sigma^{-1}) \tilde{\Sigma}}_F\le \norm{\Sigma} \cdot \norm{\tilde{\Sigma}} \cdot \norm{\tilde{\Sigma}^{-1} - \Sigma^{-1}}_F,
	\end{align*}
	so that
	\[
	\sinfdiv{f_\Sigma}{f_{\tilde{\Sigma}}} \le \norm{\Sigma}\cdot
	\norm{\tilde{\Sigma}} \cdot \norm{\tilde{\Sigma}^{-1} -
		\Sigma^{-1}}_F^2 /2 .
	\]
	Note that	since $\Sigma$ is symmetric positive definite, we have
$\norm{\Sigma} = \lambda_1(\Sigma)$.
	Write $\Sigma^{-1} = I + \Delta$  as in the proof of \lemref{psd}, so the eigenvalues of $\Sigma^{-1}$ lie in $[1- \|\Delta\|, 1+ \|\Delta\|]$.
	Therefore, 
	\[
	\norm{\Sigma} = \lambda_1(\Sigma) = \frac{1}{\lambda_d(\Sigma^{-1})} \le \frac{1}{1 - \|\Delta\|} \le \frac{1}{1 - \|\Delta\|_F} \le \frac{1}{1 - 1/2}=2 ,
	\]
	and the same bound holds for $\norm{\tilde{\Sigma}}$, whence $\sinfdiv{f_{\Sigma}}{f_{\tilde{\Sigma}}} \le 2 \|\tilde{\Sigma}^{-1} - \Sigma^{-1}\|_F^2$.
\end{proof}

To use \lemref{fano-lower} we need a lower bound on 
$\TV(f_{\Sigma}, f_{\tSigma})$. We use the following bound, which is tight up to the constant factor.

\begin{lem}[\protect{\cite[Theorem~1.1]{we_tv_gaussians}}, see also \protect{\cite[Corollary~2]{ulyanov}}]\lemlabel{tvlowerbound}
	For any pair $\Sigma,\tSigma\in\mathcal{P}_d$, we have
	\[
	\TV(f_{\Sigma}, f_{\tSigma})
	\geq 
	\min \{1,
	\| \Sigma^{1/2}\tSigma^{-1}\Sigma^{1/2} - I\|_F
	\} / 100.
	\]
\end{lem}

\begin{proof}[Proof of \propref{normal-lower-bound}]
We will use \lemref{fano-lower}.
Set $\delta = c_2/\sqrt{n}$ for a sufficiently small constant $c_2$.	
Since $n\geq c_1m$ by assumption, by choosing $c_2\leq c_1/2\sqrt2$ we can be sure that $\delta\sqrt{2m}\leq1/2$.

Note that for any $\Sigma\in\mathcal{W}$, we have
	$\|\Sigma^{-1}-I\|_F = \delta\sqrt {2m}\leq1/2$. On the one hand, 
	\lemref{normal-kl} gives that for any $\Sigma,\tSigma\in\mathcal{W}$, we have
	\[
	\infdiv{f_{\Sigma}}{f_{\tilde{\Sigma}}} \le 2
	\|\tilde{\Sigma}^{-1} - \Sigma^{-1}\|_F^2 \leq 8 m \delta^2,
	\]
	where we have used the triangle inequality for the second inequality.
	On the other hand, applying the left inequality in (\ref{frobeqs}) twice,
	\[
	\| \Sigma^{1/2}\tSigma^{-1}\Sigma^{1/2} - I\|_F
	=
	\| \Sigma^{1/2} (\tSigma^{-1}-\Sigma^{-1})
	\Sigma^{1/2} \|_F
	\geq 
	\sigma_d(\Sigma^{1/2})^2
	\|\tSigma^{-1}-\Sigma^{-1}\|_F.
	\]
	Recall that $\sigma_d(\Sigma^{1/2})$ denotes the smallest singular value of $\Sigma^{1/2}$.
	By \lemref{psd}, the eigenvalues of $\Sigma^{1/2}$ lie in $[\sqrt{2/3},\sqrt{2}]$, and since $\Sigma^{1/2}$ is symmetric and positive definite, its singular values coincide with its eigenvalues, hence $\sigma_d(\Sigma^{1/2})^2\geq 2/3.$
	Also note that 
	$\|\tSigma^{-1}-\Sigma^{-1}\|_F\geq\delta \sqrt{m/3}$ since 
	$\tSigma^{-1}$ and $\Sigma^{-1}$ differ in at least $m/3$ entries.
	Therefore, \lemref{tvlowerbound} gives 
	\begin{align*}
		\TV(f_{\Sigma}, f_{\tSigma})
		\geq
		c_3 \min \{1,
		\delta\sqrt{m}
		\} = c_3 \delta \sqrt{m}.
	\end{align*}

	Thus, in the notation of \lemref{fano-lower}, we have by
	\thmref{gilbert-varshamov} and
	the above discussion that for some $c_3> 0$,
	\[
	|\sF| \ge 2^{m/5}, \quad \alpha \ge c_3 \delta \sqrt{m}, \quad \beta \le 8 \delta^2 m.
	\]
	Recall that $\delta = c_2/\sqrt{n}$ for a sufficiently small constant $c_2$. By choosing $c_2$ small enough, we can be sure that
	\[
	1 - \frac{n \beta + \log 2}{\log |\sF|} \ge \frac{1}{2} .
	\]
	Then, by \lemref{fano-lower},
	$\sR_n(\sF_G) \ge \alpha/8 \ge (c_2 c_3/8) \sqrt{m/n}$, completing the proof.
\end{proof}

\section{Learning Ising graphical models}\seclabel{ising}
The \emph{Ising model} describes a probability distribution on the
binary hypercube $\{-1, 1\}^d$ for some $d \ge 1$, where any particular
vector $x \in \{-1, 1\}^d$ is called a \emph{configuration}.  One such
distribution is parameterized by a graph $G = ([d], E)$ with a set of
edge weights $w_{ij} \in \R$ for every edge $\{i, j\} \in E$ called
\emph{interactions}, and some weights $h_i \in \R$ for $1 \le i \le d$
called the \emph{external magnetic field}. These parameters define the
\emph{Hamiltonian} $H \colon \{-1, 1\}^d \to \R$,
\[
H(x) = \sum_{\{i, j\} \in E} w_{ij} x_i x_j + \sum_{i = 1}^d h_i x_i.
\]
Any configuration $x \in \{-1, 1\}^d$ then appears with probability
proportional to $\exp\{H(x)\}$. In fact, we can write
$H(x) = H_{h, W}(x) = x\transpose W x + h\transpose x$ for a vector
$h \in \R^d$ and a matrix $W \in \sM_G$, where
\[
\sM_G = \Big\{W \in \R^{d \times d} \colon \text{ if } \{i, j\} \not\in E \text{ with } i \neq j \in [d], \text{ then } W_{ij} = 0 \Big\} ,
\]
and in particular,
\[
W_{ij} = \left\{ \begin{array}{ll}
0 & \mbox{if $\{i, j\} \not\in E$,} \\
w_{ij}/2 & \mbox{if $\{i, j\} \in E$.}
\end{array} \right.
\]
The probability mass function of the Ising model with interactions $W$
and external magnetic field $h$ is denoted by $f_{h, W}$, where
\begin{align}
	f_{h, W}(x) = \frac{e^{H_{h, W}(x)}}{Z(h, W)} , \eqlabel{gibbs}
\end{align}
where the normalizing factor $Z(h, W)$ is called the \emph{partition
	function}, defined by
\[
Z(h, W) = \sum_{x \in \{-1, 1\}^d} e^{H_{h, W}(x)} .
\]
Probability distributions whose densities have the form \eqref{gibbs}
for general Hamiltonians are known as \emph{Gibbs distributions} or
\emph{Boltzmann distributions}.

Given a graph $G$, let $\sI_G$ be the class of all Ising
models with interactions in $\sM_G$, namely,
\[
\sI_G = \left\{ f_{h, W} \colon h\in \R^d, \, W \in \sM_G \right\} ,
\]
and let $\sI'_G$ be the subclass with no external magnetic field, namely,
\[
\sI'_G = \left\{ f_{0, W} \colon W \in \sM_G \right\} .
\]
As in \secref{normal}, $\sI_G$ is the class of all $d$-dimensional
Ising models whose components form a graphical model with respect to
$G$, and similarly for $\sI'_G$.

We omit detailed proofs of the upper bounds in
\thmref{ising-main-bound}, since they are virtually identical to that
of \thmref{normal-main-bound} as given in \secref{normal-upper}.  For
$\sI_G$, the corresponding vector space has the basis
\[
\{1\} \cup \{x_i x_j \colon \{i, j\} \in E\} \cup \{x_i \colon i\in[d]\} ,
\]
with $m + d + 1$ elements, while for $\sI'_G$, the corresponding
vector space does not have the last $d$ basis vectors, so it has
dimension $m + 1$. In the case that $m = 0$, the class $\sI'_G$
contains only one distribution (the uniform distribution on
$\{-1, 1\}^d$), and thus $\sR_n(\sI'_G) = 0$. Thus for any
$m\geq0$ and any $G$ with $m$ edges, $\sR_n(\sI'_G) \le c \sqrt{ m/n}$
for some constant $c > 0$.

\subsection{Proof of the lower bound in \thmref{ising-main-bound}~\ref{ising-2}}\seclabel{ising-construction}
Since  our Ising models in this section will have no external
magnetic field, we write $f_{W}$ for $f_{0, W}$, $H_W$ for $H_{0, W}$,
and $Z(W)$ for $Z(0, W)$.  As in \secref{normal-lower}, the lower
bound in \thmref{ising-main-bound}~\ref{ising-2} follows from the
following proposition.
\begin{prop}\proplabel{ising-lower-bound}
	There exist $c_1, c_2 > 0$ such that for any graph $G = ([d], E)$
	with $m$ edges, where $n \ge c_1 m$,
	\[
	\sR_n(\sI'_G) \ge c_2 \sqrt{m/n} .
	\]
\end{prop}
We appeal to \lemref{fano-lower} again. The construction and proof
techniques are very similar to the previous section. Indeed, let
$\sS \subseteq \{-1, 1\}^{m}$ be a set of sign vectors as in
\thmref{gilbert-varshamov}, satisfying $|\sS| \ge 2^{m/5}$ and for any
distinct $s, \tilde{s} \in \sS$, $\norm{s - \tilde{s}}_1 \ge m/3$. 
Enumerate the
edges of $G$ from $1$ to $m$, and for
$s \in \sS$, define the zero-diagonal symmetric matrix
$W(s) \in \sM_G$ with entries
\[
W(s)_{ij} = \left\{\begin{array}{ll}
0 & \mbox{if $i=j$ or $\{i, j\} \not\in E$,} \\
\delta s_{\{i, j\}} & \mbox{if $\{i, j\} \in E$.}
\end{array} \right.
\]
Then let $\sW = \{ W(s) \colon s \in \sS \}$ be a set of interactions,
and $\sI = \{f_W \colon W \in \sW\}$ be the finite class of Ising
models with interactions from $\sW$.

Now, to control the $L^1$-distance and KL-divergence between
distributions in $\sI$, a few intermediate computations are
necessary. 
First, we recall some properties of sub-gaussian random
variables.
The \emph{sub-gaussian norm} of a random variable $X$ is defined to be
\[
\|X\|_{\psi_2} = \inf\left\{ t > 0 \colon \E\{e^{(X/t)^2}\} \le 2 \right\} .
\]
A random variable $X$ is called \emph{sub-gaussian} if
$\|X\|_{\psi_2} < \infty$
(see, e.g., \cite[Section~2.5]{vershynin-book}). Observe in particular that 
any bounded random variable is sub-gaussian. Recall now the following
concentration inequality for quadratic forms of
sub-gaussian random vectors.
\begin{thm}[Hanson-Wright inequality 
	\protect{\cite[Theorem 6.2.1]{vershynin-book}},
	see also \protect{\cite[Example~2.12]{gabor-concentration}}]\thmlabel{hanson-wright}
	Let $X = (X_1, \dots, X_d)$ be a random vector with independent
	zero-mean components satisfying
	$\max_{1 \le i \le d} \|X_i\|_{\psi_2} \le K$, and let
	$W \in \R^{d \times d}$. Then, for every $t \ge 0$,
	\[
	\Pr\Big\{ |X\transpose W X - \E X\transpose W X| > t \Big\} \le 2
	\exp\left\{ - C \min \left\{ \frac{t^2}{K^4 \|W\|_F^2},
	\frac{t}{K^2 \|W\|} \right\}\right\} ,
	\]
	for some universal constant $C > 0$.
\end{thm}

A square matrix is called \emph{zero-diagonal} if all its diagonal
entries are zero.
The following moment inequalities will come in handy.

\begin{lem}\lemlabel{form-utils}
	Let $X = (X_1, \dots, X_d)$ be a random vector with i.i.d.\
	components where $\E\{X_1\} = 0$, $\E\{X_1^2\} = 1$, and $\|X_1\|_{\psi_2} \le
	K$. Let $W \in \R^{d \times d}$ be symmetric and zero-diagonal. Then,
	\begin{enumerate}[label=(\roman*)]
		\item \label{1-moment} $ \E\{X\transpose W X\} = 0 .$ 
		\item \label{2-moment} $\E\{(X\transpose W X)^2\} = 2 \|W\|_F^2 .$
		\item \label{k-moment} There exists $c_3>0$ such that for any integer $k$ we have 
		\[
		\E\{(X\transpose W X)^k\} \le c_3^k K^{2k} k! \|W\|_F^k .
		\]
		\item \label{exp-moment} There exist $c_1, c_2 > 0$ such that for any
		$t > 0$, if $c_1 K^2 t \|W\|_F\leq 1$, then
		\[
		\E\{e^{t X\transpose W X}\} \le 1 + c_2 K^4 t^2 \|W\|_F^2.
		\]
	\end{enumerate}
\end{lem}
\begin{proof}
	Observation \ref{1-moment} follows simply by writing out the quadratic form,
	\begin{align*}
		\E X\transpose W X = \sum_{i , j} W_{ij} \E\{X_i X_j\} = \sum_{i = 1}^d W_{ii} \E\{X_i^2\} + \sum_{i \neq j} W_{ij} \E\{X_i\} \E\{X_j\} = 0 .
	\end{align*}
	
	To prove \ref{2-moment}, we expand the square, and notice that only
	the monomials of the form $\E \{X_i^4\}$ or $\E\{X_i^2 X_j^2\}$ are
	nonzero after taking expectations, so
	\begin{align*}
		\E\{(X\transpose W X)^2\} & = \E\left\{ \left( \sum_{i, j} W_{ij} X_i X_j \right)^2 \right\}  \\
		& = \sum_{i = 1}^d W_{ii}^2 \E \{X_i^4\}
		+ \sum_{i\neq j} \left(W_{ij}^2 + W_{ij} W_{ji} + W_{ii}W_{jj}\right) \E \{X_i^2 X_j^2\} \\
		&= 2 \sum_{i \neq j} W_{i j}^2 = 2 \|W\|_F^2.
	\end{align*}
	
	For \ref{k-moment}, we integrate
	\begin{align*}
		\E\{(X\transpose W X)^k\} &\le \int_0^\infty \Pr\{|X\transpose W X|^k \ge t\} \dif t \\
		&\le 2 \int_0^\infty e^{-C \frac{t^{1/k}}{K^2 \|W\|}} \dif t + 2\int_0^\infty e^{-C \frac{t^{2/k}}{K^4 \|W\|_F^2}} \dif t \tag{by \thmref{hanson-wright}} \\
		&= 2 \Gamma(k + 1) \left( \frac{K^2 \|W\|}{C} \right)^k + 2 \Gamma(k/2 + 1) \left( \frac{K^4 \|W\|_F^2}{C} \right)^{k/2}  \\
		&\le c_3^k K^{2k} k! \|W\|_F^k ,
	\end{align*}
	for some $c_3 > 0$. Here, $\Gamma(t)=\int_0^{\infty} x^{t-1}e^{-x} \dif x$ is the gamma function, which is increasing on $[1,\infty)$ and satisfies $\Gamma(t+1)=t\Gamma(t)$ for any real $t$, and $\Gamma(k+1)=k!$ for any positive integer $k$.

	To prove \ref{exp-moment}, we use the power series representation
	of the exponential, so
	\begin{align*}
		\E\{e^{t X\transpose W X}\} -1 & = \sum_{k = 1}^\infty \frac{\E\{(t X\transpose W X)^k\}}{k!} & \\
		& = \sum_{k = 2}^\infty \frac{\E\{(t X\transpose W X)^k\}}{k!} & \text{($\E \{t X\transpose W X\} =0$)}\\
		& \le  \sum_{k = 2}^\infty \frac{(c_3 K^2 t \|W\|_F)^k k!}{k!}& \text{(part \ref{k-moment})}\\
		& \le 2 c_3^2 K^4 t^2 \|W\|_F^2 & \text{(if $ c_3 K^2 t \|W\|_F \leq 1/2 $)},
	\end{align*}
completing the proof.
\end{proof}

For the rest of this section, let $X = (X_1, \dots, X_d)$ denote
a uniformly random vector in $\{-1, 1\}^d$.  All expectations will be
with respect to this random variable.

\begin{lem}\lemlabel{partition-function}
	There exist $c_1, c_2 > 0$ such that for any zero-diagonal symmetric
	$W \in \R^{d \times d}$ with $\|W\|_F \le c_1$,
	\[
	1 \le 2^{-d} Z(W) \le 1 + c_2 \|W\|_F^2 .
	\]
\end{lem}
\begin{proof}
	By the definition of $Z(W)$, 
	\begin{align*}
		2^{-d} Z(W) = \sum_{x \in \{-1, 1\}^d} 2^{-d} e^{x\transpose W x} = \E\{e^{X\transpose W X}\} .
	\end{align*}
	On the one hand, by \lemref{form-utils}~\ref{1-moment}, 
	\[
	\E\{e^{X\transpose W X }\} \ge \E\{1+X\transpose W X\} = 1 ,
	\]
	and on the other hand, by \lemref{form-utils}~\ref{exp-moment},
	\begin{align*}
		\E\{e^{X\transpose W X }\} \le 1 + c_2 \|W\|_F^2,
	\end{align*}
when $\|W\|_F \leq c_1$ for some sufficiently small positive constant
	$c_1$.
\end{proof}

\begin{lem}\lemlabel{ising-kl}
	There exist $c_1, c_2 > 0$ such that for any zero-diagonal symmetric matrices
	$W, \tilde{W} \in \R^{d \times d}$ satisfying
	$\max\{\|W\|_F, \|\tilde{W}\|_F\} \le c_1$,
	\[
	\infdiv{f_W}{f_{\tilde{W}}} \le c_2 (\|W\|_F^2 + \|\tilde{W}\|_F^2) .
	\]
\end{lem}
\newcommand{\wtilde}{\widetilde{W}}
\begin{proof} 
We prove the inequality for
$\sinfdiv{f_W}{f_{\tilde{W}}} = \infdiv{f_W}{f_{\tilde{W}}} +
\infdiv{f_{\tilde{W}}}{f_W}$. By definition,
	\begin{align*}
	\infdiv{f_W}{f_{\tilde{W}}}
	& =
	2^d \E\left\{ f_W(X)  \log\left(\frac{f_W(X)}{f_{\tilde{W}}(X)}\right)\right\}
	\\& =
	2^d \E\left\{ \frac{e^{H_W(X)}}{Z(W)}  \log\left(
	\frac
	{e^{H_W(X)} Z(\wtilde)}
	{e^{H_{\wtilde}(X)} Z(W)}
	\right)\right\}.
	\end{align*}
	From \lemref{partition-function} we have $0<2^d/Z(W)\leq 1$, whence,
	\begin{equation}
		\infdiv{f_W}{f_{\tilde{W}}}
		\leq
		\E\left\{e^{H_W(X)} (H_W(X)-H_{\wtilde}(X))\right\}
		+
		\E\left\{e^{H_W(X)} 
		\log \left( \frac{Z(\wtilde)}{Z(W)} \right)
		\right\}.\label{kleq}
	\end{equation}
	We next bound the second term.
	By \lemref{form-utils}~\ref{exp-moment} and since $\|W\|_F\leq c_1$,
	$\E\left\{e^{H_W(X)}\right\} 
	\leq 1 + c_3 \|W\|_F^2 \le 1 + c_3 c_1^2$.
	From \lemref{partition-function} we have $\frac{Z(\wtilde)}{Z(W)}\leq 1 + c_4\|\wtilde\|_F^2$,
	so
	$\log\left(\frac{Z(\wtilde)}{Z(W)}\right)\leq c_4\|\wtilde\|_F^2$, thus
	\[
	\E\left\{e^{H_W(X)} 
	\log \left( \frac{Z(\wtilde)}{Z(W)} \right)
	\right\}
	\leq
	(1 + c_3 c_1^2) c_4\|\wtilde\|_F^2.
	\]
	It can be shown similarly that
	\[
	\infdiv{f_{\tilde{W}}}{f_W}
	\leq
	\E\left\{e^{H_{\tilde{W}}(X)} (H_{\wtilde}(X)-H_W(X))\right\}+c_5\|W\|_F^2,
	\] 
	hence,
	\begin{align*}
	\sinfdiv{f_W}{f_{\tilde{W}}}
	\leq 
	\E&\left\{ \left(e^{H_W(X)} - e^{H_{\tilde{W}}(X)} \right) (H_W(X) - H_{\tilde{W}}(X)) \right\} \\&+c_5 (\|W\|_F^2+\|\tilde{W}\|_F^2).
	\end{align*}
For bounding the first term, using the elementary inequality
\begin{align}
	1 + t  \le e^t \le 1 + t + \frac{t^2}{2} \max\{e^t, 1\}
	\qquad \textnormal{for any }t\in\R ,\label{exp-frieze}
\end{align}
we find that for all
$t, s \in \R$,
\[
(e^t - e^s) (t - s) \le (t - s)^2 + |t - s| \left(t^2 \max\{e^t, 1\} + s^2 \max\{e^s, 1\}\right)/2 .
\]
Using this, we get the following upper bound,
\begin{align}
	&\E\left\{ \left(e^{H_W(X)} - e^{H_{\tilde{W}}(X)} \right) (H_W(X) - H_{\tilde{W}}(X)) \right\} \notag \\
	&\qquad \le \hspace{1.1em} \E\{(H_W(X) - H_{\tilde{W}}(X))^2\} \eqlabel{ising-kl-term-1} \\
	&\qquad \quad +\, \E\left\{ |H_W(X) - H_{\tilde{W}}(X)| \, H_W(X)^2 \max\{e^{H_W(X)}, 1\}/2\right\} \eqlabel{ising-kl-term-2} \\
	&\qquad \quad +\, \E\left\{ |H_W(X) - H_{\tilde{W}}(X)| \, H_{\tilde{W}}(X)^2 \max\{e^{H_{\tilde{W}}(X)}, 1\} /2 \right\} . \eqlabel{ising-kl-term-3}
\end{align}
The term \eqref{ising-kl-term-1} is
$2 \|W - \tilde{W}\|_F^2 \le 2 (\|W\|_F + \|\tilde{W}\|_F)^2$ by
\lemref{form-utils}~\ref{2-moment} and the triangle inequality for the Frobenius norm. For bounding \eqref{ising-kl-term-2}, 
we first observe that, clearly,
$(\max\{e^{H_W(X)}, 1\}/2)^2 \leq 
e^{2 H_W(X)} + 1$.
Then, using the inequality 
\begin{equation}
	\E \{A B\} \leq \sqrt{\E \{A^2\} \E\{B^2\}},
	\eqlabel{cs}
\end{equation}
which is the Cauchy-Schwarz inequality for the inner product $\langle A, B\rangle = \E \{A B\}$,
\begin{align*}
	&\E\left\{ |H_W(X) - H_{\tilde{W}}(X)| H_W(X)^2 \max\{e^{H_W(X)}, 1\}\right\} \\
	&\qquad \le \sqrt{ \E\{ (H_W(X) - H_{\tilde{W}}(X))^2 H_W(X)^4\} (\E\{e^{2 H_W(X)}\} + 1) }.
\end{align*}
By \lemref{form-utils}~\ref{exp-moment} and since $\|W\|_F\leq c_1$ by assumption,
$\sqrt{\E\{e^{2 H_W(X)}\} + 1} \leq c_6$.
For the other factor, we apply~\eqref{cs} again:
\begin{align*}
\E\{ (H_W(X) - H_{\tilde{W}}(X))^2 H_W(X)^4\}
& \leq
\sqrt{\E\{ (H_W(X) - H_{\tilde{W}}(X))^4 \}
\E \{ H_W(X)^8 \} } \\ &
\leq c_7 \sqrt{ \| W- \wtilde\|_F^{4} \| W\|_F ^{8} }
\leq (2c_1)^2c_7 \|W\|_F^4,
\end{align*}
where the second inequality follows from
\lemref{form-utils}~\ref{k-moment}
applied to matrices $W-\wtilde$ and $W$, and the third one follows from the triangle inequality,
$\| W- \wtilde\|_F \leq \|W\|_F+\|\wtilde\|_F\leq 2c_1$.
Putting everything together, we find that
\[
\E\left\{ |H_W(X) - H_{\tilde{W}}(X)| H_W(X)^2 \max\{e^{H_W(X)}, 1\}\right\}
\leq c_6 \sqrt{(2c_1)^2c_7 } \|W\|_F^2,
\]
and a similar bound holds for \eqref{ising-kl-term-3}, after which the result follows.
\end{proof}

\begin{lem}\lemlabel{ising-l1}
	There exist $c_1, c_2, c_3 > 0$ such that for any zero-diagonal
	symmetric matrices $W, \tilde{W} \in \R^{d \times d}$ with
	$\max\{\|W\|_F, \|\tilde{W}\|_F\} \leq c_1$,
	\[
	\norm{f_W - f_{\tilde{W}}}_1 \ge c_2 \|W - \tilde{W}\|_F - c_3(\|W\|_F^2 + \|\tilde{W}\|_F^2).
	\]
\end{lem}
\begin{proof}
	By \lemref{partition-function} we have
	\[
	\left| \frac{2^d}{Z(W)}-1\right| \leq \frac{c_0 \|W\|_F^2}{1+c_0\|W\|_F^2}
	\leq c_0 \|W\|_F^2,
	\]
	so by the triangle inequality and \lemref{form-utils}~\ref{exp-moment}, there is $c_4 > 0$
	for which
	\begin{align}
		\norm{f_W - f_{\tilde{W}}}_1 &= 2^d \E\left\{ \left| \frac{e^{H_W(X)}}{Z(W)} - \frac{e^{H_{\tilde{W}}(X)}}{Z(\tilde{W})} \right| \right\} \notag \\
		&\ge \E \left\{ \left| e^{H_W(X)} - e^{H_{\tilde{W}}(X)} \right| \right\} -  c_4 (\|W\|_F^2 + \|\tilde{W}\|_F^2) . 
	\end{align}
	By~(\ref{exp-frieze})
	and the triangle inequality again,
	\begin{align*}
		\E\left\{\left|e^{H_W(X)} - e^{H_{\tilde{W}}(X)}\right|\right\} &\ge \E\{|H_W(X) - H_{\tilde{W}}(X)|\} \\
		&\hphantom{\ge} -\, (1/2) \E\left\{H_W(X)^2 \max\{e^{H_W(X)}, 1\}\right\} \\
		&\hphantom{\ge} -\, (1/2) \E\left\{H_{\tilde{W}}(X)^2 \max\{e^{H_{\tilde{W}}(X)}, 1\}\right\} .
	\end{align*}
	We first bound the second term. By~\eqref{cs} and
	\lemref{form-utils}~\ref{k-moment}, \ref{exp-moment},
	\begin{align*}
		\E\left\{H_W(X)^2 \max\{e^{H_W(X)}, 1\}\right\} \le \sqrt{\E\{H_W(X)^4\} (\E\{e^{2H_W(X)}\} + 1)} \le c_5 \|W\|_F^2,
	\end{align*}
	and a similar analysis works for the third term. For the first term,
	by H\"{o}lder's inequality and \lemref{form-utils}~\ref{2-moment},
	\ref{k-moment} (applied to the symmetric zero-diagonal matrix $W-\wtilde$), there is a $c_6 > 0$ for which
	\begin{align*}
		\E\left\{|H_W(X) - H_{\tilde{W}}(X)| \right\} &\ge \frac{\E\left\{(H_W(X) - H_{\tilde{W}}(X))^2\right\}^{3/2}}{\E\left\{(H_W(X) - H_{\tilde{W}}(X))^4\right\}^{1/2}} \ge c_6 \|W - \tilde{W}\|_F . \qedhere
	\end{align*}
\end{proof}

The proof of \propref{ising-lower-bound} is now identical to that of
\propref{normal-lower-bound}.

\subsection{Proof of the lower bound in \thmref{ising-main-bound}~\ref{ising-1}}\label{product}
Let $\overline{\mathcal I_d}$ be the class of $d$-dimensional Ising models with no
interactions. 
Note that in this case the problem is density estimation for a product distribution.

The lower bound in
\thmref{ising-main-bound}~\ref{ising-1} will follow from the next
proposition along with \thmref{ising-main-bound}~\ref{ising-2} and
subadditivity of the square root, just as in \secref{normal-lower}.
\begin{prop}\proplabel{no-external}
	There exist $c_1, c_2 > 0$ such that if $n \ge c_1 d$, 
	\[
	\sR_n(\overline{\mathcal I_d}) \ge c_2 \sqrt{d/n} .
	\]
\end{prop}
\begin{proof}[Proof sketch]
	As in the above arguments, we pick a subclass of $2^{d/5}$ densities
	of $\overline{\mathcal I_d}$ and apply \lemref{fano-lower}.  The corresponding
	magnetic fields will have entries $\pm \delta$, with the signs
	specified by \thmref{gilbert-varshamov}, so that any two of them
	differ in at least $d/6$ components.  One can then show that the
	KL-divergence between any two of these densities is at most a
	constant factor of $\delta^2 d$, while the $L^1$-distances are at
	least some constant factor of $\delta \sqrt{d}$.  The proofs are
	simpler than those in the previous section; for example, in this
	case, the partition functions can be computed exactly, and are equal
	for every density in the subclass. We omit the details.
\end{proof}

\section{Proof of the upper bound in \thmref{unknowngraph}}
\label{sec:unknowngraph}
We give the proof for $\sF_{d,m}$, and the proof for $\sI_{d,m}$ is
identical.  Let $\sG_{d, m}$ denote the set of all labeled graphs with
vertex set $[d]$ and $m$ edges.  Now, $\sF_{d,m}$ has Yatracos class
\[
\sA = \bigcup_{(G, H) \in \sG_{d, m}^2} \sA_{G, H} ,
\]
where
\[
\sA_{G, H} = \Big\{ \{ x \in \R^d \colon g(x) > h(x)\} \colon g \in \sF_G, \, h \in \sF_H \Big\} .
\]	
Note that $|\sG_{d, m}| \le \binom{\binom{d}{2}}{m} \leq d^{2m}$, and
$\VC(\sA_{G, H}) \le 2m + 2d+1$ for any $G, H \in \sG_{d, m}$, as in
the proof of the upper bound in \thmref{normal-main-bound}. By
properties of the VC-dimension of unions (see, e.g.,
\cite[Exercise~6.11]{understanding}),
\begin{align*}
	\VC(\sA) &= \VC\left( \bigcup_{(G, H) \in \sG_{d, m}^2} \sA_{G, H} \right) \\
	&\le c_1 (m+d) \log (m+d) + c_2 \log d^{4m} \\
	&\le c_3 (m + d) \log d ,
\end{align*}
so by \thmref{risk-vc},
\[ 
\sR_n(\sF_{d, m}) \le c_4 \sqrt{\frac{(m + d) \log d}{n}}.
\]

\section{Discussion}\seclabel{conclusion}
Our work raises several open problems.
\begin{enumerate}[label=\textit{\arabic*.}]
	\item \textit{Higher order forms.} We have studied estimating
	densities that are proportional to the exponential of some quadratic
	form. One can ask for the minimax risk of the class of densities in
	which this form has a higher order. Namely, let $k, d \ge 1$ be
	given integers, and suppose that $\sF$ is a class of densities
	supported on $\{-1, 1\}^d$, where each density $f \in \sF$ is
	parameterized by weights $w_{i_1, \dots, i_k} \in \R$ for each
	$1 \le i_1 < i_2 < \dots < i_k \le d$, and when $x \in \{-1, 1\}^d$,
	\[
	f(x) \propto \exp\left\{\sum_{1 \le i_1 < \dots < i_k \le d} w_{i_1, i_2, \dots,
		i_k} x_{i_1} x_{i_2} \cdots x_{i_k} \right\} .
	\]
	Then, just as in the proof of the upper bound of
	\thmref{ising-main-bound}~\ref{ising-2}, we have that there is a
	universal constant $c_1 > 0$ for which
	\[
	\sR_n(\sF) \le c_1 \min\left\{1, \sqrt{ \frac{\binom{d}{k}}{n} } \right\}.
	\]
	Can this be shown to be tight to within a constant factor? It is
	straightforward to see that the answer is \emph{yes} for $k = 1$,
	and the results of this paper show that the answer is \emph{yes} for
	$k = 2$. However, for $k \ge 3$, our techniques seem to
	fail. Auffinger and Ben Arous~\cite{auffinger} noted that when the
	weights are $w_{i_1, \dots, i_k} \iid \sN(0, 1)$, the random $k$-th
	order form $g \colon \bS^{d - 1} \to \R$ defined by
	\[
	g(x) = \sum_{i_1,  \dots, i_k = 1}^d w_{i_1, i_2, \dots, i_k} x_{i_1} x_{i_2} \dots x_{i_k} 
	\]
	blows up in complexity once $k \ge 3$. For example, they show that
	there is a $c_3 > 0$ for which $g$ has at least $e^{c_3 d}$ local
	minima on $\bS^{d - 1}$ in expectation, as long as $k \ge 3$. On the
	other hand, when $k \le 2$, deterministically $g$ has only a
	constant number of local minima on $\bS^{d - 1}$. This gap in
	complexity may indicate that analyzing the case $k \ge 3$ for our
	purposes will require  more sophisticated techniques.
	
	\item \textit{Tightness of the VC-dimension bound.}  We proved that
	$\sR_n(\sF)$ is bounded from above and below by constant factors of
	$\sqrt{\VC(\sA)/n}$, where $\sA$ is the Yatracos class of $\sF$, for
	$\sF \in \{ \sF_G, \sI_G, \sI'_G\}$. The upper bound here holds for
	any class $\sF$ by \thmref{risk-vc}, and it can be easily seen that
	there are classes of densities for which this is not tight.  Can we
	characterize the classes of densities $\sF$ for which $\sR_n(\sF)$
	is in fact on the order of $\sqrt{\VC(\sA)/n}$?
	
	\item \textit{The minimax risk of unlabeled graphical models.}  In our
	setting, the given graph $G$ is labeled, so we are given the
	specific pairs of coordinates which interact. What if only the
	structure of the graph $G$ is known, but its labeling is not? What
	if we know that $G$ is a tree?  If only the number of edges of $G$
	is known, \thmref{unknowngraph} provides some bound that is tight up
	to a factor of $\sqrt{\log d}$. Can this gap be closed?
	
	\item \textit{The minimax risk of Ising blockmodels.} For a given
	$S\subseteq [d]$ with $|S|=d/2$ and parameters
	$\alpha,\beta \in \R$, a \emph{bipartite mean-field model}~\cite{bipartitemeanfield}, also called an \emph{Ising blockmodel}~\cite{berthet}, has density
	\[
	f_{S,\alpha,\beta}(x) = \exp\left\{ \frac{\beta}{2d} \sum_{i \sim j} x_i x_j + \frac{\alpha}{2d} \sum_{i \not\sim j} x_i x_j \right\} \bigg / Z(\alpha,\beta) 
	\]
	for $x \in \{-1, 1\}^d$, where $i \sim j$ means that either
	$i, j \in S$ or $i, j \not\in S$, and $i \not\sim j$ means that one
	of $i, j$ is in $S$ and one is not, and $Z(\alpha,\beta)$ is the
	normalizing factor. Motivated by social network
	analysis and the notion of communities in such networks, Berthet, Rigollet,
	and Srivastava~\cite{berthet} studied this model in a learning context. Their work is mainly concerned with the
	estimation or recovery of $S$ from $n$ independent samples of
	$f_{S, \alpha, \beta}$, but one can also ask for the minimax
	learning rate for this class of densities, if some or all of
	$\alpha, \beta$ and $S$ are unknown.
	
	\item \textit{Efficient algorithms for density estimation.}
	Here, we focused on statistical efficiency rather than computational efficiency. 
	As the VC-based density estimate in \thmref{risk-vc} cannot be efficiently computed, a natural open problem is to design efficient density estimators achieving the minimax risks for normal and Ising distributions.
	As far as we know, only for the class of multivariate normal distributions, which correspond to normal graphical models with respect to the complete graph,
	such an efficient algorithm exists. It is based on careful mean and covariance estimations (see~\cite[Appendix~B]{2017-abbas}).
\end{enumerate}

\section*{Acknowledgements}
We thank the anonymous referee of the Electronic Journal of Statistics for their comments on improving the presentation.


\begin{thebibliography}{45}
	
	\bibitem{acharya-lower-bound}
	\begin{bincollection}[author]
		\bauthor{\bsnm{Acharya},~\bfnm{Jayadev}\binits{J.}},
		\bauthor{\bsnm{Jafarpour},~\bfnm{Ashkan}\binits{A.}},
		\bauthor{\bsnm{Orlitsky},~\bfnm{Alon}\binits{A.}} \AND
		\bauthor{\bsnm{Suresh},~\bfnm{Ananda~Theertha}\binits{A.~T.}}
		(\byear{2014}).
		\btitle{Near-optimal-sample estimators for spherical {G}aussian mixtures}.
		In \bbooktitle{Advances in Neural Information Processing Systems 27}
		(\beditor{\bfnm{Z.}\binits{Z.}~\bsnm{Ghahramani}},
		\beditor{\bfnm{M.}\binits{M.}~\bsnm{Welling}},
		\beditor{\bfnm{C.}\binits{C.}~\bsnm{Cortes}},
		\beditor{\bfnm{N.~D.}\binits{N.~D.}~\bsnm{Lawrence}} \AND
		\beditor{\bfnm{K.~Q.}\binits{K.~Q.}~\bsnm{Weinberger}}, eds.)
		\bpages{1395--1403}.
		\bpublisher{Curran Associates, Inc.}
	\end{bincollection}
	\endbibitem
	
	\bibitem{2017-abbas}
	\begin{barticle}[author]
		\bauthor{\bsnm{Ashtiani},~\bfnm{H.}\binits{H.}},
		\bauthor{\bsnm{Ben-David},~\bfnm{S.}\binits{S.}},
		\bauthor{\bsnm{Harvey},~\bfnm{N.}\binits{N.}},
		\bauthor{\bsnm{Liaw},~\bfnm{C.}\binits{C.}},
		\bauthor{\bsnm{Mehrabian},~\bfnm{A.}\binits{A.}} \AND
		\bauthor{\bsnm{Plan},~\bfnm{Y.}\binits{Y.}}
		(\byear{2018}).
		\btitle{Near-optimal sample complexity bounds for robust learning of
			{G}aussians mixtures via compression schemes}.
		\bjournal{ArXiv e-prints}.
		\bnote{Appeared in the proceedings of Neural Information Processing Systems
			2018, available at \url{https://arxiv.org/abs/1710.05209v3}}.
	\end{barticle}
	\endbibitem
	
	\bibitem{2020-abbas}
	\begin{barticle}[author]
		\bauthor{\bsnm{Ashtiani},~\bfnm{H.}\binits{H.}},
		\bauthor{\bsnm{Ben-David},~\bfnm{S.}\binits{S.}},
		\bauthor{\bsnm{Harvey},~\bfnm{N.}\binits{N.}},
		\bauthor{\bsnm{Liaw},~\bfnm{C.}\binits{C.}},
		\bauthor{\bsnm{Mehrabian},~\bfnm{A.}\binits{A.}} \AND
		\bauthor{\bsnm{Plan},~\bfnm{Y.}\binits{Y.}}
		(\byear{2019}).
		\btitle{Near-optimal sample complexity bounds for robust learning of
			{G}aussians mixtures via compression schemes}.
		\bjournal{ArXiv e-prints}.
		\bnote{To appear in the Journal of the ACM, available at
			\url{https://arxiv.org/abs/1710.05209v4}}.
	\end{barticle}
	\endbibitem
	
	\bibitem{abbas-mixtures}
	\begin{binproceedings}[author]
		\bauthor{\bsnm{Ashtiani},~\bfnm{Hassan}\binits{H.}},
		\bauthor{\bsnm{Ben-David},~\bfnm{Shai}\binits{S.}} \AND
		\bauthor{\bsnm{Mehrabian},~\bfnm{Abbas}\binits{A.}}
		(\byear{2018}).
		\btitle{Sample-efficient learning of mixtures}.
		In \bbooktitle{Proceedings of the Thirty-Second AAAI Conference on Artificial
			Intelligence}.
		\bseries{AAAI '18}
		\bpages{2679--2686}.
		\bpublisher{AAAI Publications.}
		\bnote{Available at \url{https://arxiv.org/abs/1706.01596}}.
	\end{binproceedings}
	\endbibitem
	
	\bibitem{assouad}
	\begin{barticle}[author]
		\bauthor{\bsnm{Assouad},~\bfnm{Patrice}\binits{P.}}
		(\byear{1983}).
		\btitle{Deux remarques sur l'estimation}.
		\bjournal{C. R. Acad. Sci. Paris S\'er. I Math.}
		\bvolume{296}
		\bpages{1021--1024}.
	\end{barticle}
	\endbibitem
	
	\bibitem{auffinger}
	\begin{barticle}[author]
		\bauthor{\bsnm{Auffinger},~\bfnm{Antonio}\binits{A.}} \AND
		\bauthor{\bsnm{Ben~Arous},~\bfnm{Gerard}\binits{G.}}
		(\byear{2013}).
		\btitle{Complexity of random smooth functions on the high-dimensional sphere}.
		\bjournal{Ann. Probab.}
		\bvolume{41}
		\bpages{4214--4247}.
		\bdoi{10.1214/13-AOP862}
	\end{barticle}
	\endbibitem
	
	\bibitem{ulyanov}
	\begin{barticle}[author]
		\bauthor{\bsnm{{Barsov}},~\bfnm{S.~S.}\binits{S.~S.}} \AND
		\bauthor{\bsnm{{Ul'yanov}},~\bfnm{V.~V.}\binits{V.~V.}}
		(\byear{1987}).
		\btitle{{Estimates of the proximity of Gaussian measures.}}
		\bjournal{{Sov. Math., Dokl.}}
		\bvolume{34}
		\bpages{462--466}.
	\end{barticle}
	\endbibitem
	
	\bibitem{berthet}
	\begin{barticle}[author]
		\bauthor{\bsnm{Berthet},~\bfnm{Quentin}\binits{Q.}},
		\bauthor{\bsnm{Rigollet},~\bfnm{Philippe}\binits{P.}} \AND
		\bauthor{\bsnm{Srivastava},~\bfnm{Piyush}\binits{P.}}
		(\byear{2019}).
		\btitle{Exact recovery in the {I}sing blockmodel}.
		\bjournal{Ann. Statist.}
		\bvolume{47}
		\bpages{1805--1834}.
		\bdoi{10.1214/17-AOS1620}
	\end{barticle}
	\endbibitem
	
	\bibitem{gabor-concentration}
	\begin{bbook}[author]
		\bauthor{\bsnm{Boucheron},~\bfnm{St{\'e}phane}\binits{S.}},
		\bauthor{\bsnm{Lugosi},~\bfnm{G{\'a}bor}\binits{G.}} \AND
		\bauthor{\bsnm{Massart},~\bfnm{Pascal}\binits{P.}}
		(\byear{2013}).
		\btitle{Concentration Inequalities: A Nonasymptotic Theory of Independence}.
		\bpublisher{Oxford University Press, Oxford}.
	\end{bbook}
	\endbibitem
	
	\bibitem{Bresler}
	\begin{binproceedings}[author]
		\bauthor{\bsnm{Bresler},~\bfnm{Guy}\binits{G.}}
		(\byear{2015}).
		\btitle{Efficiently learning {I}sing models on arbitrary graphs}.
		In \bbooktitle{Proceedings of the Forty-Seventh Annual ACM Symposium on Theory
			of Computing}.
		\bseries{STOC ’15}
		\bpages{771--782}.
		\bpublisher{Association for Computing Machinery}, \baddress{New York, NY, USA}.
		\bdoi{10.1145/2746539.2746631}
	\end{binproceedings}
	\endbibitem
	
	\bibitem{costis-2018}
	\begin{binproceedings}[author]
		\bauthor{\bsnm{Daskalakis},~\bfnm{Constantinos}\binits{C.}},
		\bauthor{\bsnm{Dikkala},~\bfnm{Nishanth}\binits{N.}} \AND
		\bauthor{\bsnm{Kamath},~\bfnm{Gautam}\binits{G.}}
		(\byear{2018}).
		\btitle{Testing {I}sing models}.
		In \bbooktitle{Proceedings of the Twenty-Ninth Annual ACM-SIAM Symposium on
			Discrete Algorithms}
		\bpages{1989-2007}.
		\bdoi{10.1137/1.9781611975031.130}
	\end{binproceedings}
	\endbibitem
	
	\bibitem{devroye-course}
	\begin{bbook}[author]
		\bauthor{\bsnm{Devroye},~\bfnm{Luc}\binits{L.}}
		(\byear{1987}).
		\btitle{A Course in Density Estimation}.
		\bseries{Progress in Probability and Statistics}
		\bvolume{14}.
		\bpublisher{Birkh\"{a}user Boston, Inc., Boston, MA}.
	\end{bbook}
	\endbibitem
	
	\bibitem{devroye-gyorfi}
	\begin{bbook}[author]
		\bauthor{\bsnm{Devroye},~\bfnm{Luc}\binits{L.}} \AND
		\bauthor{\bsnm{Gy\"orfi},~\bfnm{L\'aszl\'o}\binits{L.}}
		(\byear{1985}).
		\btitle{Nonparametric Density Estimation: The $L_1$ View}.
		\bseries{Wiley Series in Probability and Mathematical Statistics: Tracts on
			Probability and Statistics}.
		\bpublisher{John Wiley \& Sons, Inc., New York}.
	\end{bbook}
	\endbibitem
	
	\bibitem{devroye-lugosi}
	\begin{bbook}[author]
		\bauthor{\bsnm{Devroye},~\bfnm{Luc}\binits{L.}} \AND
		\bauthor{\bsnm{Lugosi},~\bfnm{G\'{a}bor}\binits{G.}}
		(\byear{2001}).
		\btitle{Combinatorial Methods in Density Estimation}.
		\bseries{Springer Series in Statistics}.
		\bpublisher{Springer-Verlag, New York}.
	\end{bbook}
	\endbibitem
	
	\bibitem{we_tv_gaussians}
	\begin{barticle}[author]
		\bauthor{\bsnm{Devroye},~\bfnm{Luc}\binits{L.}},
		\bauthor{\bsnm{Mehrabian},~\bfnm{Abbas}\binits{A.}} \AND
		\bauthor{\bsnm{Reddad},~\bfnm{Tommy}\binits{T.}}
		(\byear{2019}).
		\btitle{The total variation distance between high-dimensional {G}aussians}.
		\bjournal{ArXiv e-prints}.
		\bnote{Available at \url{https://arxiv.org/abs/1810.08693}}.
	\end{barticle}
	\endbibitem
	
	\bibitem{Diakonikolas2016}
	\begin{bincollection}[author]
		\bauthor{\bsnm{Diakonikolas},~\bfnm{Ilias}\binits{I.}}
		(\byear{2016}).
		\btitle{Learning structured distributions}.
		In \bbooktitle{Handbook of Big Data}.
		\bseries{Chapman \& Hall/CRC Handb. Mod. Stat. Methods}
		\bpages{267--283}.
		\bpublisher{CRC Press, Boca Raton, FL}.
	\end{bincollection}
	\endbibitem
	
	\bibitem{dudley_vectorspace}
	\begin{barticle}[author]
		\bauthor{\bsnm{Dudley},~\bfnm{R.~M.}\binits{R.~M.}}
		(\byear{1978}).
		\btitle{Central limit theorems for empirical measures}.
		\bjournal{Ann. Probab.}
		\bvolume{6}
		\bpages{899--929}.
		\bdoi{10.1214/aop/1176995384}
	\end{barticle}
	\endbibitem
	
	\bibitem{bipartitemeanfield}
	\begin{barticle}[author]
		\bauthor{\bsnm{Fedele},~\bfnm{Micaela}\binits{M.}} \AND
		\bauthor{\bsnm{Unguendoli},~\bfnm{Francesco}\binits{F.}}
		(\byear{2012}).
		\btitle{Rigorous results on the bipartite mean-field model}.
		\bjournal{Journal of Physics A: Mathematical and Theoretical}
		\bvolume{45}
		\bpages{385001}.
		\bdoi{10.1088/1751-8113/45/38/385001}
	\end{barticle}
	\endbibitem
	
	\bibitem{gilbert}
	\begin{barticle}[author]
		\bauthor{\bsnm{Gilbert},~\bfnm{Edgar~N.}\binits{E.~N.}}
		(\byear{1952}).
		\btitle{A comparison of signalling alphabets}.
		\bjournal{Bell System Tech. J.}
		\bvolume{31}
		\bpages{504--522}.
	\end{barticle}
	\endbibitem
	
	\bibitem{deeplearningbook}
	\begin{bbook}[author]
		\bauthor{\bsnm{Goodfellow},~\bfnm{Ian}\binits{I.}},
		\bauthor{\bsnm{Bengio},~\bfnm{Yoshua}\binits{Y.}} \AND
		\bauthor{\bsnm{Courville},~\bfnm{Aaron}\binits{A.}}
		(\byear{2016}).
		\btitle{Deep Learning}.
		\bseries{Adaptive Computation and Machine Learning}.
		\bpublisher{MIT Press, Cambridge, MA}.
	\end{bbook}
	\endbibitem
	
	\bibitem{Hamilton-upper}
	\begin{bincollection}[author]
		\bauthor{\bsnm{Hamilton},~\bfnm{Linus}\binits{L.}},
		\bauthor{\bsnm{Koehler},~\bfnm{Frederic}\binits{F.}} \AND
		\bauthor{\bsnm{Moitra},~\bfnm{Ankur}\binits{A.}}
		(\byear{2017}).
		\btitle{Information theoretic properties of {M}arkov random fields, and their
			algorithmic applications}.
		In \bbooktitle{Advances in Neural Information Processing Systems 30}
		(\beditor{\bfnm{I.}\binits{I.}~\bsnm{Guyon}},
		\beditor{\bfnm{U.~V.}\binits{U.~V.}~\bsnm{Luxburg}},
		\beditor{\bfnm{S.}\binits{S.}~\bsnm{Bengio}},
		\beditor{\bfnm{H.}\binits{H.}~\bsnm{Wallach}},
		\beditor{\bfnm{R.}\binits{R.}~\bsnm{Fergus}},
		\beditor{\bfnm{S.}\binits{S.}~\bsnm{Vishwanathan}} \AND
		\beditor{\bfnm{R.}\binits{R.}~\bsnm{Garnett}}, eds.)
		\bpages{2463--2472}.
		\bpublisher{Curran Associates, Inc.}
	\end{bincollection}
	\endbibitem
	
	\bibitem{has-fano}
	\begin{barticle}[author]
		\bauthor{\bsnm{Has'minski\u\i},~\bfnm{R.~Z.}\binits{R.~Z.}}
		(\byear{1978}).
		\btitle{A lower bound for risks of nonparametric density estimates in the
			uniform metric}.
		\bjournal{Teor. Veroyatnost. i Primenen.}
		\bvolume{23}
		\bpages{824--828}.
	\end{barticle}
	\endbibitem
	
	\bibitem{linearalgebrahandbook}
	\begin{bbook}[author]
		\beditor{\bsnm{Hogben},~\bfnm{Leslie}\binits{L.}}, ed.
		(\byear{2014}).
		\btitle{Handbook of linear algebra},
		\bedition{second} ed.
		\bseries{Discrete Mathematics and its Applications (Boca Raton)}.
		\bpublisher{CRC Press, Boca Raton, FL}.
	\end{bbook}
	\endbibitem
	
	\bibitem{matrix_analysis}
	\begin{bbook}[author]
		\bauthor{\bsnm{Horn},~\bfnm{Roger~A.}\binits{R.~A.}} \AND
		\bauthor{\bsnm{Johnson},~\bfnm{Charles~R.}\binits{C.~R.}}
		(\byear{2013}).
		\btitle{Matrix Analysis},
		\bedition{Second} ed.
		\bpublisher{Cambridge University Press, Cambridge}.
	\end{bbook}
	\endbibitem
	
	\bibitem{ibragimov}
	\begin{bincollection}[author]
		\bauthor{\bsnm{Ibragimov},~\bfnm{Ildar}\binits{I.}}
		(\byear{2001}).
		\btitle{Estimation of analytic functions}.
		In \bbooktitle{State of the Art in Probability and Statistics ({L}eiden,
			1999)}.
		\bseries{IMS Lecture Notes Monogr. Ser.}
		\bvolume{36}
		\bpages{359--383}.
		\bpublisher{Inst. Math. Statist., Beachwood, OH}.
		\bdoi{10.1214/lnms/1215090078}
	\end{bincollection}
	\endbibitem
	
	\bibitem{KMV}
	\begin{barticle}[author]
		\bauthor{\bsnm{Kalai},~\bfnm{Adam}\binits{A.}},
		\bauthor{\bsnm{Moitra},~\bfnm{Ankur}\binits{A.}} \AND
		\bauthor{\bsnm{Valiant},~\bfnm{Gregory}\binits{G.}}
		(\byear{2012}).
		\btitle{Disentangling {G}aussians}.
		\bjournal{Comm. ACM}
		\bvolume{55}.
	\end{barticle}
	\endbibitem
	
	\bibitem{Kearns}
	\begin{binproceedings}[author]
		\bauthor{\bsnm{Kearns},~\bfnm{Michael}\binits{M.}},
		\bauthor{\bsnm{Mansour},~\bfnm{Yishay}\binits{Y.}},
		\bauthor{\bsnm{Ron},~\bfnm{Dana}\binits{D.}},
		\bauthor{\bsnm{Rubinfeld},~\bfnm{Ronitt}\binits{R.}},
		\bauthor{\bsnm{Schapire},~\bfnm{Robert~E.}\binits{R.~E.}} \AND
		\bauthor{\bsnm{Sellie},~\bfnm{Linda}\binits{L.}}
		(\byear{1994}).
		\btitle{On the learnability of discrete distributions}.
		In \bbooktitle{Proceedings of the Twenty-sixth Annual ACM Symposium on Theory
			of Computing}.
		\bseries{STOC '94}
		\bpages{273--282}.
		\bpublisher{ACM}, \baddress{New York, NY, USA}.
		\bdoi{10.1145/195058.195155}
	\end{binproceedings}
	\endbibitem
	
	\bibitem{Klivans-upper}
	\begin{bincollection}[author]
		\bauthor{\bsnm{Klivans},~\bfnm{Adam~R.}\binits{A.~R.}} \AND
		\bauthor{\bsnm{Meka},~\bfnm{Raghu}\binits{R.}}
		(\byear{2017}).
		\btitle{Learning graphical models using multiplicative weights}.
		In \bbooktitle{58th {A}nnual {IEEE} {S}ymposium on {F}oundations of {C}omputer
			{S}cience---{FOCS} 2017}
		\bpages{343--354}.
		\bpublisher{IEEE Computer Soc., Los Alamitos, CA}.
	\end{bincollection}
	\endbibitem
	
	\bibitem{kullback-book}
	\begin{bbook}[author]
		\bauthor{\bsnm{Kullback},~\bfnm{Solomon}\binits{S.}}
		(\byear{1997}).
		\btitle{Information Theory and Statistics}.
		\bpublisher{Dover Publications, Inc., Mineola, NY.}
		\bnote{Reprint of the second (1968) edition}.
	\end{bbook}
	\endbibitem
	
	\bibitem{graphical-models}
	\begin{bbook}[author]
		\bauthor{\bsnm{Lauritzen},~\bfnm{S.~L.}\binits{S.~L.}}
		(\byear{1996}).
		\btitle{Graphical Models}.
		\bseries{Oxford Statistical Science Series}
		\bvolume{17}.
		\bpublisher{The Clarendon Press, Oxford University Press, New York}.
	\end{bbook}
	\endbibitem
	
	\bibitem{lecam-1}
	\begin{barticle}[author]
		\bauthor{\bsnm{Le~Cam},~\bfnm{Lucien}\binits{L.}}
		(\byear{1973}).
		\btitle{Convergence of estimates under dimensionality restrictions}.
		\bjournal{Ann. Statist.}
		\bvolume{1}
		\bpages{38--53}.
	\end{barticle}
	\endbibitem
	
	\bibitem{lecam-2}
	\begin{bbook}[author]
		\bauthor{\bsnm{Le~Cam},~\bfnm{Lucien}\binits{L.}}
		(\byear{1986}).
		\btitle{Asymptotic Methods in Statistical Decision Theory}.
		\bseries{Springer Series in Statistics}.
		\bpublisher{Springer-Verlag, New York}.
	\end{bbook}
	\endbibitem
	
	\bibitem{Santhanam-lower}
	\begin{barticle}[author]
		\bauthor{\bsnm{Santhanam},~\bfnm{Narayana~P.}\binits{N.~P.}} \AND
		\bauthor{\bsnm{Wainwright},~\bfnm{Martin~J.}\binits{M.~J.}}
		(\byear{2012}).
		\btitle{Information-theoretic limits of selecting binary graphical models in
			high dimensions}.
		\bjournal{IEEE Trans. Inform. Theory}
		\bvolume{58}
		\bpages{4117--4134}.
		\bdoi{10.1109/TIT.2012.2191659}
	\end{barticle}
	\endbibitem
	
	\bibitem{understanding}
	\begin{bbook}[author]
		\bauthor{\bsnm{Shalev-Shwartz},~\bfnm{Shai}\binits{S.}} \AND
		\bauthor{\bsnm{Ben-David},~\bfnm{Shai}\binits{S.}}
		(\byear{2014}).
		\btitle{Understanding Machine Learning: From Theory to Algorithms}.
		\bpublisher{Cambridge University Press}
		\bnote{Available at
			\url{http://www.cs.huji.ac.il/~shais/UnderstandingMachineLearning/copy.html}}.
	\end{bbook}
	\endbibitem
	
	\bibitem{Shanmugam-lower}
	\begin{binproceedings}[author]
		\bauthor{\bsnm{Shanmugam},~\bfnm{Karthikeyan}\binits{K.}},
		\bauthor{\bsnm{Tandon},~\bfnm{Rashish}\binits{R.}},
		\bauthor{\bsnm{Dimakis},~\bfnm{Alexandros~G.}\binits{A.~G.}} \AND
		\bauthor{\bsnm{Ravikumar},~\bfnm{Pradeep}\binits{P.}}
		(\byear{2014}).
		\btitle{On the information theoretic limits of learning {I}sing models}.
		In \bbooktitle{Proceedings of the 27th International Conference on Neural
			Information Processing Systems}.
		\bseries{NIPS'14}
		\bvolume{2}
		\bpages{2303--2311}.
		\bpublisher{MIT Press}, \baddress{Cambridge, MA, USA}.
	\end{binproceedings}
	\endbibitem
	
	\bibitem{talagrand-2003}
	\begin{bbook}[author]
		\bauthor{\bsnm{Talagrand},~\bfnm{Michel}\binits{M.}}
		(\byear{2003}).
		\btitle{Spin Glasses: A Challenge for Mathematicians. Cavity and Mean Field
			Models}.
		\bseries{Ergebnisse der Mathematik und ihrer Grenzgebiete. 3. Folge. A Series
			of Modern Surveys in Mathematics}
		\bvolume{46}.
		\bpublisher{Springer-Verlag, Berlin}.
	\end{bbook}
	\endbibitem
	
	\bibitem{talagrand-2010}
	\begin{bbook}[author]
		\bauthor{\bsnm{Talagrand},~\bfnm{Michel}\binits{M.}}
		(\byear{2011}).
		\btitle{Mean Field Models for Spin Glasses. {V}olume {I}. Basic Examples}.
		\bseries{Ergebnisse der Mathematik und ihrer Grenzgebiete. 3. Folge. A Series
			of Modern Surveys in Mathematics}
		\bvolume{54}.
		\bpublisher{Springer-Verlag, Berlin}.
		\bdoi{10.1007/978-3-642-15202-3}
	\end{bbook}
	\endbibitem
	
	\bibitem{talagrand-2011}
	\begin{bbook}[author]
		\bauthor{\bsnm{Talagrand},~\bfnm{Michel}\binits{M.}}
		(\byear{2011}).
		\btitle{Mean Field Models for Spin Glasses. {V}olume {II}. Advanced
			Replica-Symmetry and Low Temperature}.
		\bseries{Ergebnisse der Mathematik und ihrer Grenzgebiete. 3. Folge. A Series
			of Modern Surveys in Mathematics}
		\bvolume{55}.
		\bpublisher{Springer, Heidelberg}.
	\end{bbook}
	\endbibitem
	
	\bibitem{tsybakov}
	\begin{bbook}[author]
		\bauthor{\bsnm{Tsybakov},~\bfnm{Alexandre~B.}\binits{A.~B.}}
		(\byear{2009}).
		\btitle{Introduction to Nonparametric Estimation}.
		\bseries{Springer Series in Statistics}.
		\bpublisher{Springer, New York}
		\bnote{Revised and extended from the 2004 French original, Translated by
			Vladimir Zaiats}.
		\bdoi{10.1007/b13794}
	\end{bbook}
	\endbibitem
	
	\bibitem{vapnik-cherv}
	\begin{barticle}[author]
		\bauthor{\bsnm{Vapnik},~\bfnm{V.~N.}\binits{V.~N.}} \AND
		\bauthor{\bsnm{\v{C}ervonenkis},~\bfnm{A.~Ja.}\binits{A.~J.}}
		(\byear{1971}).
		\btitle{The uniform convergence of frequencies of the appearance of events to
			their probabilities}.
		\bjournal{Teor. Verojatnost. i Primenen.}
		\bvolume{16}
		\bpages{264--279}.
	\end{barticle}
	\endbibitem
	
	\bibitem{varshamov}
	\begin{barticle}[author]
		\bauthor{\bsnm{Var\v{s}amov},~\bfnm{R.~R.}\binits{R.~R.}}
		(\byear{1957}).
		\btitle{The evaluation of signals in codes with correction of errors}.
		\bjournal{Dokl. Akad. Nauk}
		\bvolume{117}
		\bpages{739--741}.
	\end{barticle}
	\endbibitem
	
	\bibitem{vershynin-book}
	\begin{bbook}[author]
		\bauthor{\bsnm{Vershynin},~\bfnm{Roman}\binits{R.}}
		(\byear{2018}).
		\btitle{High-dimensional probability}.
		\bseries{Cambridge Series in Statistical and Probabilistic Mathematics}
		\bvolume{47}.
		\bpublisher{Cambridge University Press, Cambridge}
		\bnote{Available at
			\url{https://www.math.uci.edu/~rvershyn/papers/HDP-book/HDP-book.html}}.
		\bdoi{10.1017/9781108231596}
	\end{bbook}
	\endbibitem
	
	\bibitem{VMLC}
	\begin{binproceedings}[author]
		\bauthor{\bsnm{Vuffray},~\bfnm{Marc}\binits{M.}},
		\bauthor{\bsnm{Misra},~\bfnm{Sidhant}\binits{S.}},
		\bauthor{\bsnm{Lokhov},~\bfnm{Andrey~Y.}\binits{A.~Y.}} \AND
		\bauthor{\bsnm{Chertkov},~\bfnm{Michael}\binits{M.}}
		(\byear{2016}).
		\btitle{Interaction Screening: Efficient and Sample-Optimal Learning of Ising
			Models}.
		In \bbooktitle{Proceedings of the 30th International Conference on Neural
			Information Processing Systems}.
		\bseries{NIPS’16}
		\bpages{2603–2611}.
		\bpublisher{Curran Associates Inc.}, \baddress{Red Hook, NY, USA}.
	\end{binproceedings}
	\endbibitem
	
	\bibitem{yatracos}
	\begin{barticle}[author]
		\bauthor{\bsnm{Yatracos},~\bfnm{Yannis~G.}\binits{Y.~G.}}
		(\byear{1988}).
		\btitle{A note on {$L_1$} consistent estimation}.
		\bjournal{Canad. J. Statist.}
		\bvolume{16}
		\bpages{283--292}.
		\bdoi{10.2307/3314734}
	\end{barticle}
	\endbibitem
	
	\bibitem{yu-survey}
	\begin{bincollection}[author]
		\bauthor{\bsnm{Yu},~\bfnm{Bin}\binits{B.}}
		(\byear{1997}).
		\btitle{{A}ssouad, {F}ano, and {L}e {C}am}.
		In \bbooktitle{Festschrift for {L}ucien {L}e {C}am}
		\bpages{423--435}.
		\bpublisher{Springer, New York}.
	\end{bincollection}
	\endbibitem
	
\end{thebibliography}
\end{document}